\documentclass[reqno,10pt]{article}

\usepackage{a4wide}
\usepackage{color}
\usepackage{mathrsfs}
\usepackage{mathtools}
\usepackage{amsmath}
\usepackage{amsthm}
\usepackage{amssymb}
\usepackage{bbm}
\usepackage{tikz-cd}
\usepackage{esint}
\usepackage{nicefrac}
\numberwithin{equation}{section}
\usepackage[colorlinks,citecolor=blue,linkcolor=red]{hyperref}
\usepackage{longtable}
\usepackage{setspace}

\usepackage{comment}

\usepackage[latin1]{inputenc}

\newcommand{\nchi}{{\raise.3ex\hbox{\(\chi\)}}}
\newcommand{\N}{\mathbb{N}}
\newcommand{\R}{\mathbb{R}}
\newcommand{\B}{\mathbb{B}}

\newcommand{\D}{{\rm D}}

\newcommand{\sfd}{{\sf d}}
\renewcommand{\d}{{\mathrm d}}
\newcommand{\e}{{\rm e}}
\newcommand{\X}{{\rm X}}
\newcommand{\Y}{{\rm Y}}

\newcommand{\1}{\mathbbm 1}

\newcommand{\LIP}{{\rm LIP}}

\newcommand{\lip}{{\rm lip}}

\newcommand{\ppi}{{\mbox{\boldmath\(\pi\)}}}

\newcommand{\sppi}{{\mbox{\scriptsize\boldmath\(\pi\)}}}

\renewcommand{\div}{{\rm div}}

\newcommand{\Der}{{\rm Der}}

\newcommand{\fr}{\penalty-20\null\hfill\(\blacksquare\)}

\newtheorem{theorem}{Theorem}[section]
\newtheorem{corollary}[theorem]{Corollary}
\newtheorem{lemma}[theorem]{Lemma}
\newtheorem{proposition}[theorem]{Proposition}
\newtheorem{definition}[theorem]{Definition}

\newtheorem{remark}[theorem]{Remark}

\title{Vector calculus on weighted reflexive Banach spaces}
\author{Enrico Pasqualetto\footnote{\href{mailto:enrico.e.pasqualetto@jyu.fi}{enrico.e.pasqualetto@jyu.fi},
Department of Mathematics and Statistics, P.O.\ Box 35 (MaD), FI-40014 University of Jyv\"askyl\"a, Finland.}
\;and Tapio Rajala\footnote{\href{mailto:tapio.m.rajala@jyu.fi}{tapio.m.rajala@jyu.fi},
Department of Mathematics and Statistics, P.O.\ Box 35 (MaD), FI-40014 University of Jyv\"askyl\"a, Finland.}}
\begin{document}

\date{\today}
\maketitle
\begin{abstract}
We study first-order Sobolev spaces on reflexive Banach spaces via relaxation, test plans, and divergence.
We show the equivalence of the different approaches to the Sobolev spaces and to the related tangent bundles.
\end{abstract}

\section{Introduction}
During the past couple of decades, Sobolev spaces \(W^{1,p}(\X,\mu)\) have been extensively studied for metric measure spaces \((\X,\sfd,\mu)\).
In spaces satisfying a local Poincar\'e inequality and measure doubling, called PI-spaces, the strong density of Lipschitz functions \cite{Cheeger00}
implies quite straightforwardly the equivalence of different approaches to Sobolev spaces. In \cite{AmbrosioGigliSavare11-3}, it was noticed that the
density in energy of Lipschitz functions, valid without the PI-assumption, is enough for showing the equivalence of notions of upper gradients. 
In general metric measure spaces one can also introduce an abstract first-order differentiable structure \cite{Gigli14}. With more assumptions on the space,
the structure is given via Lipschitz charts \cite{Cheeger00} (see also \cite{EBS2023}). 

If the underlying metric structure is linear, a natural question is to ask if we can connect the abstract differentiable structures with the linear one. 
In Euclidean spaces with arbitrary reference measure this was addressed in \cite{Louet14} and \cite{LPR21} (see also \cite{GL22} for the BV case) by starting
from the notions of Sobolev space given in \cite{BouchitteButtazzoSeppecher1997} and \cite{Zhi00}. Although weighted Euclidean spaces need not be PI-spaces,
smooth functions are still strongly dense in the Sobolev space defined on them \cite{GP16-2}.

In the current work we consider the infinite-dimensional linear case of reflexive separable Banach spaces (and, in some results, the larger class of separable Banach spaces
having the Radon--Nikod\'{y}m property). In this context, it is known from \cite{Sav19} (see also \cite{FornasierSavareSodini22,Sodini22}) that suitable
algebras \(\mathscr A\) of smooth functions on \(\B\) are dense in energy in the Sobolev space \(W^{1,p}(\B,\mu)\). One advantage of working with smooth
functions is that they have an everywhere-defined differential, which allows to transfer valuable information from the ambient space to the Sobolev space with respect to arbitrary reference measures.
Notice that even though our results are stated in the infinite-dimensional case, they are new already on Euclidean spaces since we allow the exponent $p \ne 2$
in $W^{1,p}$. Indeed, only the case $p=2$ was handled in \cite{LPR21} since there the arguments relied heavily on the Hilbertian structure of the Sobolev space $W^{1,2}$.
Similarly, it would be simpler to prove versions of our results in cases where one could use the Hilbertian structure, such as for $W^{1,2}$ on infinite-dimensional Hilbert spaces,
or for the cases where one can renorm the Sobolev space to be Hilbert (compare to \cite{EBRS2022b}).

The paper contains three main results, which are valid for weighted Banach spaces (as in Definition \ref{def:weighted_Banach}) that are reflexive
(or some weaker assumptions, as the Radon--Nikod\'{y}m property):
\begin{itemize}
\item In Theorem \ref{thm:main_Sob} we prove that the metric notion of Sobolev space via upper gradients (see Definition \ref{def:Sobolev_space}) coincides
with other two approaches (in terms of vector fields having distributional divergence, and via weak differentials) that are tailored to the Banach setting.
\item In Theorem \ref{thm:main_mwug} we show that also the minimal weak upper gradients corresponding to the three above approaches do coincide.
\item Theorem \ref{thm:main_tg_bundle} connects different approaches to identifying the directions in the Banach space that are analytically relevant for \(W^{1,p}(\B,\mu)\).
The first approach is via divergence inspired by Bouchitt\'e--Buttazzo--Seppecher \cite{BouchitteButtazzoSeppecher1997} while the other two are via test plans. This result
partly generalises the Euclidean one \cite[Theorem 3.16]{LPR21}.
\end{itemize}
\subsection*{List of symbols}
Below, we provide a list of the non-standard symbols we will use in the paper.
\begin{center}
\begin{spacing}{1.2}
\begin{longtable}{p{2.1cm} p{11.8cm}}
\(\mathcal L_1\) & Lebesgue measure restricted to \([0,1]\); see \eqref{eq:def_restr_Leb}.\\
\((\B,\mu)\) & a weighted Banach space; see Definition \ref{def:weighted_Banach}.\\
\(\mathfrak C(\B)\) & shorthand notation for \(C([0,1];\B)\times[0,1]\); see \eqref{eq:frak_C(X)}.\\
\(\e\) & the evaluation map \(\e\colon\mathfrak C(\X)\to\X\), given by \(\e(\gamma,t)=\e_t(\gamma)\coloneqq\gamma_t\).\\
\(\Der\) & derivative map; see \eqref{eq:def_Der}.\\
\(|\d f|_{\B^*}\) & the function \(\B\ni x\mapsto\|\d_x f\|_{\B^*}\in\R\) for \(f\in C^1(\B)\); see \eqref{eq:def_ptwse_norm_Fr_diff}.\\
\({\rm Comp}(\ppi)\) & compression constant of a \(q\)-test plan \(\ppi\); see Definition \ref{def:test_plan}.\\
\(\Pi_q(\B,\mu)\) & \(q\)-test plans on the weighted Banach space \((\B,\mu)\); see Definition \ref{def:test_plan}.\\
\(\hat\ppi\) & shorthand notation for \(\hat\ppi\coloneqq\ppi\otimes\mathcal L_1\); see \eqref{eq:aux_def_hat_ppi}.\\
\(\{\hat\ppi_x\}_{x\in\B}\) & conditional probabilities of the disintegration \(\hat\ppi=\int\hat\ppi_x\,\d(\e_\#\hat\ppi)(x)\).\\
\(W^{1,p}(\B,\mu)\) & metric \(p\)-Sobolev space on a weighted Banach space \((\B,\mu)\); see Definition \ref{def:Sobolev_space}.\\
\(|\D_\mu f|\) & the minimal \(p\)-weak upper gradient of \(f\in W^{1,p}(\B,\mu)\); see Definition \ref{def:Sobolev_space}.\\
\({\rm S}_\sppi\) & the `support' of a \(q\)-test plan \(\ppi\); see \eqref{eq:def_S_ppi}.\\
\(D_q(\div_\mu)\) & domain of the distributional divergence; see Definition \ref{def:distr_div}.\\
\(\div_\mu\) & the distributional divergence of \(v\in D_q(\div_\mu)\); see Definition \ref{def:distr_div}.\\
\(\mathscr D_\mu(\B)\) & space of \(\mu\)-a.e.\ defined measurable \(\B\)-bundles; see Definition \ref{def:bundle}.\\
\(\Gamma_q({\bf E})\) & \(q\)-section space of a measurable \(\B\)-bundle \({\bf E}\in\mathscr D_\mu(\B)\); see \eqref{eq:def_sect_bundle}.\\
\(\preceq\) & the natural partial order on \(\mathscr D_\mu(\B)\).\\
\(T_\mu\B\) & \(q\)-tangent bundle of \((\B,\mu)\); see Definition \ref{def:tg_bundle}.\\
\(\Gamma_p({\bf E}_{w^*}^*)\) & space of weakly\(^*\) measurable sections of the dual of \(\bf E\).\\
\(|\omega|_{{\bf E}^*}\) & \(\mu\)-a.e.\ equivalence class of \(\X\ni x\mapsto\|\omega(x)\|_{{\bf E}(x)^*}\in\R\) for \(\omega\in\Gamma_p({\bf E}_{w^*}^*)\).\\
\(L^p_{w^*}(\X,\mu;\B^*)\) & the dual space of \(L^q(\X,\mu;\B)\).\\
\(\omega_{|{\bf E}}\) & the restriction of \(\omega\in\Gamma_p({\bf E}_{w^*}^*)\) to a bundle \({\bf E}\in\mathscr D_\mu(\B)\); see Definition \ref{def:restr_Fr_diff}.\\
\(\d_{|{\bf E}}\) & the restricted Fr\'{e}chet differential \(C^1(\B)\cap\LIP_b(\B)\ni f\mapsto\d f_{|{\bf E}}\); see \eqref{eq:restr_diff}.\\
\([\omega]_{{\bf E}^*}\) & the pointwise seminorm of \(\omega\in L^p(\X,\mu;\B^*)\) induced by \({\bf E}\in\mathscr D_\mu(\B)\); see \eqref{eq:ptwse_seminorm}.\\
\(\pi_{\bf E}\) & the `restriction' operator \(L^p(\X,\mu;\B^*)\ni\omega\mapsto\omega_{|{\bf E}}\in\Gamma_p({\bf E}_{w^*}^*)\); see \eqref{eq:restr_op}.\\
\(\varphi^*{\bf E}\) & the pullback bundle; see \eqref{eq:def_pullback_bundle}.\\
\(L_f\) & the functional \(L_f\in D_q(\div_\mu)^*\) associated to \(f\in W^{1,p}(\B,\mu)\); see \eqref{eq:char_L_f_cl1}.\\
\({\rm WD}_\mu^p(f;\mathscr A)\) & \((\mathscr A,p)\)-weak differentials of \(f\in W^{1,p}(\B,\mu)\); see Theorem \ref{thm:main_Sob} c).\\
\(\d_\mu\) & the \(p\)-weak differential operator \(\d_\mu\colon W^{1,p}(\B,\mu)\to\Gamma_q(T_\mu\B)^*\); see \eqref{eq:def_d_mu}.\\
\(v_\sppi\) & vector field induced by a \(q\)-test plan \(\ppi\); see Lemma \ref{lem:vf_induced_tp}.\\
\(S_\sppi\B\) & multivalued map sending \(x\in\B\) to the support of \(\Der_\#\hat\ppi_x\); see \eqref{eq:char_master_tp_aux}.\\
\(V_\sppi\B\) & the bundle obtained as the closure of the span of \(S_\sppi\B\); see Lemma \ref{lem:char_master_tp_aux}.\\
\(\Der_\sppi\) & equivalence class of \(\Der\) in \(L^q(\mathfrak C(\B),\hat\ppi;\B)\); see the proof of Theorem \ref{thm:main_tg_bundle}.\\
\(\ppi'\) & velocity field of a \(q\)-test plan \(\ppi\); see Remark \ref{rmk:vel_tp}.
\end{longtable}
\end{spacing}
\end{center}
\vspace{-1.2cm}
\paragraph{Acknowledgements.} 
The authors thank the anonymous referee, whose comments and suggestions led to a significant improvement of the presentation.
The first named author has been supported by the MIUR-PRIN 202244A7YL project ``Gradient Flows and Non-Smooth Geometric Structures with Applications to Optimization and Machine Learning''.
\section{Preliminaries}\label{s:preli}
Throughout the whole paper, whenever an exponent \(p\in(1,\infty)\) is given, we tacitly denote by
\[
q\coloneqq\frac{p}{p-1}\in(1,\infty)
\]
its conjugate exponent, and vice versa. Moreover, letting \(\mathcal L^1\) be the Lebesgue measure, we shorten
\begin{equation}\label{eq:def_restr_Leb}
\mathcal L_1\coloneqq\mathcal L^1|_{[0,1]}.
\end{equation}
In this paper, we focus on the family of \emph{weighted Banach spaces}, which we define as follows:
\begin{definition}[Weighted Banach space]\label{def:weighted_Banach}
We say that a couple \((\B,\mu)\) is a \textbf{weighted Banach space} if \(\B\) is a separable Banach space and \(\mu\geq 0\)
is a finite Borel measure on \(\B\).
\end{definition}

We underline that in the above definition we assume that \(\B\) is \underline{separable} and \(\mu\) is \underline{finite}.
\subsection{Classical notions on Banach spaces}
Given a normed space \(\mathbb V\), we denote by \(\mathbb V^*\) its dual Banach space. If \(\B\) is the Banach space obtained as the completion of \(\mathbb V\),
then we have that \(\B^*\cong\mathbb V^*\), meaning that \(\B^*\) can be canonically identified with \(\mathbb V^*\). The duality pairing between \(\omega\in\mathbb V^*\)
and \(v\in\mathbb V\) will be denoted by \(\langle\omega,v\rangle\in\R\).
\subsubsection*{Absolutely continuous curves}
Let \(\B\) be a separable Banach space. We denote by \(C([0,1];\B)\) the space of all continuous curves \(\gamma\colon[0,1]\to\B\).
It is a complete and separable metric space if endowed with the following distance:
\[
\sfd_{C([0,1];\B)}(\gamma,\sigma)\coloneqq\max_{t\in[0,1]}\|\gamma_t-\sigma_t\|_\B\quad\text{ for every }\gamma,\sigma\in C([0,1];\B).
\]
We then define the complete and separable metric space \((\mathfrak C(\B),\sfd_{\mathfrak C(\B)})\) as
\begin{equation}\label{eq:frak_C(X)}
\mathfrak C(\B)\coloneqq C([0,1];\B)\times[0,1]
\end{equation}
together with the distance
\[
\sfd_{\mathfrak C(\B)}((\gamma,t),(\sigma,s))\coloneqq\sfd_{C([0,1];\B)}(\gamma,\sigma)+|t-s|\quad\text{ for every }(\gamma,t),(\sigma,s)\in\mathfrak C(\B).
\]

The \textbf{evaluation map} \(\e\colon\mathfrak C(\B)\to\B\) is defined as \(\e(\gamma,t)=\e_t(\gamma)\coloneqq\gamma_t\)
for every \((\gamma,t)\in\mathfrak C(\B)\). Notice that \(\e\colon\mathfrak C(\B)\to\B\) is continuous and \(\e_t\colon C([0,1];\B)\to\B\) is \(1\)-Lipschitz
for every \(t\in[0,1]\). A curve \(\gamma\in C([0,1];\B)\) is said to be \textbf{\(q\)-absolutely continuous} (for some \(q\in[1,\infty]\)) if there exists
a function \(g\in L^q(0,1)\) such that \(g\geq 0\) and
\begin{equation}\label{eq:def_AC}
\|\gamma_t-\gamma_s\|_\B\leq\int_s^t g(r)\,\d r\quad\text{ for every }0\leq s<t\leq 1.
\end{equation}
In the case where \(q=1\), we write `absolutely continuous' instead of `\(1\)-absolutely continuous'.
We denote by \(AC^q([0,1];\B)\) the family of all \(q\)-absolutely continuous curves in \(\B\).
It holds that \(AC^q([0,1];\B)\) is a Borel subset of \(C([0,1];\B)\). We write \(AC([0,1];\B)\) instead of \(AC^1([0,1];\B)\).
\subsubsection*{Radon--Nikod\'{y}m property and Asplund spaces}
A Banach space \(\B\) is said to have the \textbf{Radon--Nikod\'{y}m property} provided every absolutely continuous curve
\(\gamma\colon[0,1]\to\B\) is \(\mathcal L_1\)-a.e.\ differentiable, which means that
\[
\dot\gamma_t\coloneqq\lim_{h\to 0}\frac{\gamma_{t+h}-\gamma_t}{h}\in\B\quad\text{ exists for }\mathcal L_1\text{-a.e.\ }t\in[0,1].
\]
By an \textbf{Asplund space} we mean a Banach space \(\B\) whose dual \(\B^*\) has
the Radon--Nikod\'{y}m property. Recall that every reflexive Banach space is Asplund and has the Radon--Nikod\'{y}m property.
The converse can fail: there exist (separable) Asplund spaces having the Radon--Nikod\'{y}m property that are not reflexive, e.g.\ James' space \cite{James51}.
See \cite{FHHMZ10} for a thorough treatment of these topics.
\medskip

Given a Banach space \(\B\) having the Radon--Nikod\'{y}m property, we define \(\Der\colon\mathfrak C(\B)\to\B\) as
\begin{equation}\label{eq:def_Der}
\Der(\gamma,t)\coloneqq\left\{\begin{array}{ll}
\dot\gamma_t\\
0_\B
\end{array}\quad\begin{array}{ll}
\text{ if }\gamma\in AC([0,1];\B)\text{ and }\dot\gamma_t\text{ exists,}\\
\text{otherwise.}
\end{array}\right.
\end{equation}
Then \(\Der\colon\mathfrak C(\B)\to\B\) is a Borel map. For any \(\gamma\in AC([0,1];\B)\) the function
\(\|\Der(\gamma,\cdot)\|_\B\), which is called the \textbf{metric speed} of \(\gamma\), is the \(\mathcal L_1\)-a.e.\ minimal
\(g\in L^1(0,1)\) with \(g\geq 0\) satisfying \eqref{eq:def_AC}.
\subsubsection*{Fr\'{e}chet differential and smooth functions}
Given a Banach space \(\B\) and a function \(f\colon\B\to\R\), we say that \(f\) is \textbf{Fr\'{e}chet differentiable} at \(x\in\B\)
if there exists an element \(\d_x f\in\B^*\), called the \textbf{Fr\'{e}chet differential} of \(f\) at \(x\), such that
\begin{equation}\label{eq:def_Frechet_diff}
\lim_{\B\ni v\to 0}\frac{|f(x+v)-f(x)-\langle\d_x f,v\rangle|}{\|v\|_\B}=0.
\end{equation}
Notice that \eqref{eq:def_Frechet_diff} determines uniquely \(\d_x f\) and implies that \(f\) is continuous at \(x\). Moreover, we say that \(f\) is \textbf{of class \(C^1\)}
if it is Fr\'{e}chet differentiable at every point of \(\B\) and \(\B\ni x\mapsto\d_x f\in\B^*\) is continuous. We denote by \(C^1(\B)\) the space of real-valued
functions of class \(C^1\) defined on \(\B\). Given any \(f\in C^1(\B)\), we define the Borel function \(|\d f|_{\B^*}\colon\B\to[0,+\infty)\) as
\begin{equation}\label{eq:def_ptwse_norm_Fr_diff}
|\d f|_{\B^*}(x)\coloneqq\|\d_x f\|_{\B^*}\quad\text{ for every }x\in\B.
\end{equation}
One can readily check that the function \(f\) is locally Lipschitz and it holds that \(|\d f|_{\B^*}=\lip(f)\), where the \textbf{slope} \(\lip(f)\colon\B\to[0,+\infty)\)
is defined as \(\lip(f)(x)\coloneqq 0\) if \(x\in\B\) is an isolated point and
\[
\lip(f)(x)\coloneqq\limsup_{\B\ni y\to x}\frac{|f(x)-f(y)|}{\|x-y\|_\B}\quad\text{ if }x\in\B\text{ is an accumulation point.}
\]
\subsubsection*{Lebesgue--Bochner spaces}
Given a finite measure space \((\X,\Sigma,\mu)\) and an exponent \(p\in[1,\infty]\), we denote by \((L^p(\mu),\|\cdot\|_{L^p(\mu)})\)
the \textbf{Lebesgue space} of exponent \(p\). Recall that \(L^p(\mu)\) is a Riesz space if endowed with the natural partial
order relation: given any two functions \(f,g\in L^p(\mu)\), we declare that
\(f\leq g\) if and only if \(f(x)\leq g(x)\) holds for \(\mu\)-a.e.\ \(x\in\X\). Recall that the Riesz space \(L^p(\mu)\)
is Dedekind complete, which means that every non-empty subset of \(L^p(\mu)\) that is bounded above has a supremum. Namely,
given a set \(\{f_i\}_{i\in I}\subseteq L^p(\mu)\) and \(g\in L^p(\mu)\) such that \(f_i\leq g\) for every \(i\in I\),
then the \textbf{supremum}
\[
f\coloneqq\bigvee_{i\in I}f_i\in L^p(\mu)
\]
exists. This means that \(f\geq f_i\) for every \(i\in I\), and that \(f\leq\tilde f\) whenever
\(\tilde f\in L^p(\mu)\) satisfies \(\tilde f\geq f_i\) for every \(i\in I\). In a similar way, one can define the
\textbf{infimum} \(\bigwedge_{i\in I}f_i\in L^p(\mu)\). See e.g.\ \cite{Bogachev07}.
\medskip

We assume the reader is familiar with the basics of Bochner integration; we refer to \cite{HNVW16}
and the references therein for a detailed account of this theory. Let us only recall some notation
and results. Given a finite measure space \((\X,\Sigma,\mu)\), a Banach space \(\B\),
and an exponent \(q\in(1,\infty)\), we denote by \(L^q(\X,\mu;\B)\) the \textbf{\(q\)-Lebesgue--Bochner space}
from \((\X,\Sigma,\mu)\) to \(\B\). The following hold:
\begin{itemize}
\item \(L^q(\X,\mu;\B)\) is a Banach space and a module over the commutative ring \(L^\infty(\mu)\).
\item \(L^p(\X,\mu;\B^*)\) is isomorphic to a subspace of \(L^q(\X,\mu;\B)^*\).
\item \(L^q(\X,\mu;\B)^*\cong L^p(\X,\mu;\B^*)\) if and only if \(\B\) is Asplund.
\item \(L^q(\X,\mu;\B)\) is reflexive if and only if \(\B\) is reflexive.
\item \(L^q(\X,\mu;\B)\) is uniformly convex if and only if \(\B\) is uniformly convex.
\end{itemize}
Given any \(v\in L^q(\X,\mu;\B)\), the \(\mu\)-a.e.\ equivalence class \(|v|_\B\) of the function \(\X\ni x\mapsto\|v(x)\|_\B\) belongs to \(L^q(\mu)\).
If \((\X,\Sigma_\X,\mu_\X)\) and \((\Y,\Sigma_\Y,\mu_\Y)\) are finite measure spaces, then each measurable map
\(\varphi\colon\X\to\Y\) satisfying \(\varphi_\#\mu_\X\leq C\mu_\Y\) for some \(C>0\) induces a pullback operator
\begin{equation}\label{eq:pullback_Leb-Boch}
\varphi^*\colon L^q(\Y,\mu_\Y;\B)\to L^q(\X,\mu_\X;\B)
\end{equation}
for every Banach space \(\B\) and \(q\in(1,\infty)\). Namely, we set \(\varphi^*v\coloneqq v\circ\varphi\) for all \(v\in L^q(\Y,\mu_\Y;\B)\).
\medskip

Given a weighted Banach space \((\B,\mu)\) and \(f\in C^1(\B)\cap\LIP(\B)\), the map \(\B\ni x\mapsto\d_x f\in\B^*\) is Bochner integrable
(as it is continuous and bounded), thus we can consider its equivalence class
\[
\d f\in L^p(\B,\mu;\B^*)\quad\text{ for every }p\in(1,\infty).
\]
For any \(v\in L^q(\B,\mu;\B)\), the \(\mu\)-a.e.\ equivalence class \(\d f(v)\) of \(\B\ni x\mapsto\langle\d_x f,v(x)\rangle\in\R\) is in \(L^1(\mu)\).
\subsection{Sobolev calculus on weighted Banach spaces}
\subsubsection*{Test plans}
Following \cite{AmbrosioGigliSavare11}, we give the ensuing definition of a \emph{\(q\)-test plan}
(over a weighted Banach space):
\begin{definition}[Test plan]\label{def:test_plan}
Let \((\B,\mu)\) be a weighted Banach space such that \(\B\) has the Radon--Nikod\'{y}m property and
let \(q\in(1,\infty)\). Then a Borel probability measure \(\ppi\) on \(C([0,1];\B)\) is said to be a
\textbf{\(q\)-test plan} on \((\B,\mu)\) provided the following two requirements are met:
\begin{itemize}
\item[\(\rm i)\)] There exists a constant \(C>0\) such that
\[
(\e_t)_\#\ppi\leq C\mu\quad\text{ for every }t\in[0,1].
\]
The minimal such \(C\) is called the \textbf{compression constant} of \(\ppi\) and denoted by \({\rm Comp}(\ppi)\).
\item[\(\rm ii)\)] The measure \(\ppi\) is concentrated on \(AC^q([0,1];\B)\) and has finite \textbf{kinetic \(q\)-energy}, i.e.
\[
\int\!\!\!\int_0^1\|\dot\gamma_t\|_\B^q\,\d t\,\d\ppi(\gamma)<+\infty.
\]
\end{itemize}
We denote by \(\Pi_q(\B,\mu)\) the family of all \(q\)-test plans on \((\B,\mu)\).
\end{definition}

We introduce the shorthand notation
\begin{equation}\label{eq:aux_def_hat_ppi}
\hat\ppi\coloneqq\ppi\otimes\mathcal L_1\quad\text{ for every }\ppi\in\Pi_q(\B,\mu).
\end{equation}
Observe that \(\e_\#\hat\ppi\leq{\rm Comp}(\ppi)\mu\), so that in particular \(\e_\#\hat\ppi\ll\mu\), for every \(\ppi\in\Pi_q(\B,\mu)\).
We will occasionally consider the \textbf{disintegration} \(\hat\ppi=\int\hat\ppi_x\,\d(\e_\#\hat\ppi)(x)\) of \(\hat\ppi\) along \(\e\), which means that:
\begin{itemize}
\item \(\{\hat\ppi_x\}_{x\in\B}\) are Borel probability measures on \(\mathfrak C(\B)\) such that
\[
\hat\ppi_x(\mathfrak C(\B)\setminus\e^{-1}(x))=0\quad\text{ for }\e_\#\hat\ppi\text{-a.e.\ }x\in\B.
\]
\item \(\B\ni x\mapsto\hat\ppi_x\) is measurable, i.e.\ \(\B\ni x\mapsto\hat\ppi_x(E)\) is Borel for every
\(E\subseteq\mathfrak C(\B)\) Borel.
\item \(\hat\ppi(E)=\int\hat\ppi_x(E)\,\d(\e_\#\hat\ppi)(x)\) for every \(E\subseteq\mathfrak C(\B)\) Borel.
\end{itemize}
The family \(\{\hat\ppi_x\}_{x\in\B}\) is \(\e_\#\hat\ppi\)-a.e.\ unique. For a proof of its existence, see e.g.\ \cite[Theorem 5.3.1]{AmbrosioGigliSavare08}.
\subsubsection*{Compatible algebras}
For any Banach space \(\B\), we call \(\LIP_b(\B)\) the algebra of real-valued bounded Lipschitz functions
on \(\B\). Following \cite[Definition 2.1.17]{Sav19}, we give the ensuing definition of \emph{compatible algebra}:
\begin{definition}[Compatible algebra]
Let \(\B\) be a Banach space. Then we say that a set \(\mathscr A\) is a \textbf{compatible subalgebra}
of \(\LIP_b(\B)\) provided it is a unital subalgebra of \(\LIP_b(\B)\) such that
\[
\|x-y\|_\B=\sup\Big\{|f(x)-f(y)|\;\Big|\;f\in\LIP_b(\B)\text{ is \(1\)-Lipschitz}\Big\}\quad\text{ for every }x,y\in\B.
\]
\end{definition}

Distinguished examples of compatible subalgebras are \(\LIP_b(\B)\) itself, \(C^1(\B)\cap\LIP_b(\B)\),
and the smaller space of \textbf{smooth cylindrical functions} on \(\B\) (defined in \cite[Example 2.1.19]{Sav19}).
A useful fact concerning compatible algebras, which is proved in \cite[Lemma 2.1.27]{Sav19}, states the following:
\begin{lemma}
Let \((\B,\mu)\) be a weighted Banach space. Let \(\mathscr A\) be a compatible subalgebra of \(\LIP_b(\B)\).
Then \(\mathscr A\) is strongly dense in \(L^p(\mu)\) for every \(p\in[1,\infty)\), and is weakly\(^*\) dense
in \(L^\infty(\mu)\).
\end{lemma}
\subsubsection*{Metric Sobolev spaces}
Let us recall the definition of the metric Sobolev space via test plans introduced in \cite{AmbrosioGigliSavare11}:
\begin{definition}[Metric Sobolev space]\label{def:Sobolev_space}
Let \((\B,\mu)\) be a weighted Banach space such that \(\B\) has the Radon--Nikod\'{y}m property and let \(p\in(1,\infty)\).
Then we say that \(f\in L^p(\mu)\) is a \textbf{\(p\)-Sobolev function} provided there exists \(G\in L^p(\mu)\) with
\(G\geq 0\) such that
\begin{equation}\label{eq:minimal_wug}
\int f(\gamma_1)-f(\gamma_0)\,\d\ppi(\gamma)\leq\int\!\!\!\int_0^1 G(\gamma_t)\|\dot\gamma_t\|_\B\,\d t\,\d\ppi(\gamma)
\quad\text{ for every }\ppi\in\Pi_q(\B,\mu).
\end{equation}
The \(\mu\)-a.e.\ minimal \(G\) verifying \eqref{eq:minimal_wug} is called the \textbf{minimal \(p\)-weak upper gradient} of \(f\)
and is denoted by \(|\D_\mu f|\in L^p(\mu)\). The space of all \(p\)-Sobolev functions on \((\B,\mu)\) is denoted by \(W^{1,p}(\B,\mu)\).
\end{definition}

The Sobolev space \(W^{1,p}(\B,\mu)\) is a Banach space if endowed with the following norm:
\[
\|f\|_{W^{1,p}(\B,\mu)}\coloneqq\Big(\|f\|_{L^p(\mu)}^p+\||\D_\mu f|\|_{L^p(\mu)}^p\Big)^{1/p}
\quad\text{ for every }f\in W^{1,p}(\B,\mu).
\]
It holds that \(\LIP_b(\B)\subseteq W^{1,p}(\B,\mu)\) and \(|\D_\mu f|\leq\lip(f)\) for every \(f\in\LIP_b(\B)\).
Moreover, the minimal \(p\)-weak upper gradient \(|\D_\mu f|\) of any given function \(f\in W^{1,p}(\B,\mu)\) can
be equivalently characterised as the \(\mu\)-a.e.\ minimal \(G\in L^p(\mu)\) with \(G\geq 0\) satisfying the following
property: for every \(\ppi\in\Pi_q(\B,\mu)\), we have that \(f\circ\gamma\in W^{1,1}(0,1)\) holds for
\(\ppi\)-a.e.\ \(\gamma\in C([0,1];\B)\) and
\begin{equation}\label{eq:ptwse_mwug}
|(f\circ\gamma)'_t|\leq G(\gamma_t)\|\dot\gamma_t\|_\B\quad\text{ for }\hat\ppi\text{-a.e.\ }(\gamma,t)\in\mathfrak C(\B).
\end{equation}

The following approximation result (which closes the gap with Cheeger's approach to metric Sobolev spaces \cite{Cheeger00},
based on a relaxation procedure) was proved in \cite[Theorem 5.2.7]{Sav19} after \cite{AmbrosioGigliSavare11-3}:
\begin{theorem}[Density in energy of compatible algebras]
Let \((\B,\mu)\) be a weighted Banach space such that \(\B\) has the Radon--Nikod\'{y}m property and let \(p\in(1,\infty)\).
Let \(\mathscr A\) be a compatible subalgebra of \(\LIP_b(\B)\). Let \(W^{1,p}(\B,\mu)\) be a given function. Then there exists
a sequence \((f_n)_n\subseteq\mathscr A\) such that
\[
f_n\to f,\quad|\d f_n|_{\B^*}\to|\D_\mu f|\quad\text{ strongly in }L^p(\mu).
\]
\end{theorem}

Another proof of the density in energy of smooth cylindrical functions (and thus of every subalgebra of \(\LIP_b(\B)\) containing them)
has been recently obtained in \cite{LucPas24}. Whereas the results of \cite{AmbrosioGigliSavare11-3,Sav19} are for arbitrary metric
measure spaces and rely on metric tools, the arguments in \cite{LucPas24} are tailored for weighted Banach spaces and are based on a
purely smooth analysis.
\medskip

It is unknown whether \(\LIP_b(\B)\) is strongly dense in \(W^{1,p}(\B,\mu)\). A sufficient condition for the strong
density of all compatible algebras is the reflexivity of \(W^{1,p}(\B,\mu)\), see \cite[Proposition 42]{ACM14}.
Examples of non-reflexive Sobolev spaces are known (see \cite[Proposition 44]{ACM14} and \cite[Section 12.5]{Heinonen07}).
We also point out that if \(W^{1,p}(\B,\mu)\) is reflexive, then it is separable (see again \cite[Proposition 42]{ACM14})
and every compatible subalgebra of \(\LIP_b(\B)\) is strongly dense in it.
\begin{remark}{\rm
It is shown in \cite{DiMarinoSpeight13} that minimal \(p\)-weak upper gradients depend on \(p\), in the sense that
for \(f\in W^{1,p}(\B,\mu)\cap W^{1,\tilde p}(\B,\mu)\) with \(p\neq\tilde p\) it can happen that the minimal
\(p\)-weak upper gradient of \(f\) differs from its minimal \(\tilde p\)-weak upper gradient. Nevertheless, in our
notation \(|\D_\mu f|\) we do not specify the exponent \(p\), since the latter will be always clear from the context.
\fr}\end{remark}
\subsubsection*{Master test plans}
Let \((\B,\mu)\) be a weighted Banach space such that \(\B\) has the Radon--Nikod\'{y}m property. Fix an exponent \(q\in(1,\infty)\).
To any \(q\)-test plan \(\ppi\in\Pi_q(\B,\mu)\) we associate the Borel set \({\rm S}_\sppi\subseteq\B\) given by
\begin{equation}\label{eq:def_S_ppi}
{\rm S}_\sppi\coloneqq\bigg\{x\in\B\;\bigg|\;\frac{\d(\e_\#\hat\ppi)}{\d\mu}(x)>0\bigg\}.
\end{equation}
The set \({\rm S}_\sppi\) is uniquely determined up to \(\mu\)-a.e.\ null sets, and
\(\mu|_{{\rm S}_\sppi}\ll\e_\#\hat\ppi\) and \((\e_\#\hat\ppi)|_{\B\setminus{\rm S}_\sppi}=0\).
\medskip

It is proved in \cite[Theorem 2.6]{Pasqualetto22} that one can always find a \textbf{master \(q\)-test plan} \(\ppi\)
on \((\B,\mu)\), i.e.\ a \(q\)-test plan having the following property: for any \(f\in W^{1,p}(\B,\mu)\), the function
\(|\D_\mu f|\in L^p(\mu)\) is the minimal \(G\in L^p(\mu)\) with \(G\geq 0\) such that \eqref{eq:ptwse_mwug} holds.
See also \cite[Theorem A.2]{GigliNobili} for an alternative proof of the existence of a master \(q\)-test plan.
\medskip

Moreover, it is proved in \cite[Proposition 2]{Pasqualetto23} that a given \(\ppi\in\Pi_q(\B,\mu)\) is a master \(q\)-test plan
if and only if for every function \(f\in W^{1,p}(\B,\mu)\) it holds that \(|\D_\mu f|=0\) \(\mu\)-a.e.\ on \(\B\setminus{\rm S}_\sppi\) and
\begin{equation}\label{eq:equiv_master_tp}
|\D_\mu f|(x)=\underset{\hat\sppi_x\text{-a.e.\ }(\gamma,t)}{\rm ess\,sup}\1_{\{\Der\neq 0\}}(\gamma,t)
\frac{|(f\circ\gamma)'_t|}{\|\dot\gamma_t\|_\B}\quad\text{ for }\e_\#\hat\ppi\text{-a.e.\ }x\in\B,
\end{equation}
where \(\hat\ppi=\int\hat\ppi_x\,\d(\e_\#\hat\ppi)(x)\) denotes the disintegration of \(\hat\ppi\) along \(\e\).
\subsubsection*{Distributional divergence}
Testing against smooth functions, we can define the space of vector fields with \(\mu\)-divergence:
\begin{definition}[Distributional \(\mu\)-divergence]\label{def:distr_div}
Let \((\B,\mu)\) be a weighted Banach space and let \(q\in(1,\infty)\). Then we define \(D_q(\div_\mu)\subseteq L^q(\B,\mu;\B)\) as the space of all vector fields \(v\in L^q(\B,\mu;\B)\)
for which there exists a function \(\div_\mu(v)\in L^q(\mu)\), called the \textbf{\(\mu\)-divergence} of \(v\), such that
\[
\int\d f(v)\,\d\mu=-\int f\,\div_\mu(v)\,\d\mu\quad\text{ for every }f\in C^1(\B)\cap\LIP_b(\B).
\]
The function \(\div_\mu(v)\in L^q(\mu)\) is uniquely determined by the density of \(C^1(\B)\cap\LIP_b(\B)\) in \(L^p(\mu)\).
\end{definition}

Some comments on the space of vector fields with distributional \(\mu\)-divergence are in order:
\begin{itemize}
\item \(D_q(\div_\mu)\) is a vector subspace of \(L^q(\B,\mu;\B)\) and \(\div_\mu\colon D_q(\div_\mu)\to L^q(\mu)\) is a linear operator.
\item The \textbf{Leibniz rule} holds, i.e.\ for every \(h\in C^1(\B)\cap\LIP_b(\B)\)
and \(v\in D_q(\div_\mu)\) we have that
\begin{equation}\label{eq:Leibniz_div_mu}
hv\in D_q(\div_\mu),\quad\div_\mu(hv)=h\,\div_\mu(v)+\d h(v).
\end{equation}
Indeed, \(\int\d f(hv)\,\d\mu=\int\d(fh)(v)\,\d\mu-\int f\,\d h(v)\,\d\mu=-\int f(h\,\div_\mu(v)+\d h(v))\,\d\mu\)
holds for every \(f\in C^1(\B)\cap\LIP_b(\B)\), whence the validity of \eqref{eq:Leibniz_div_mu} follows.
\item In particular, the space \(D_q(\div_\mu)\) is a \(C^1(\B)\cap\LIP_b(\B)\)-submodule (thus, a vector subspace) of \(L^q(\B,\mu;\B)\).
Since \(C^1(\B)\cap\LIP_b(\B)\) is a weakly\(^*\) dense subalgebra of \(L^\infty(\mu)\), it follows that
\({\rm cl}_{L^q(\B,\mu;\B)}(D_q(\div_\mu))\) is a \(L^\infty(\mu)\)-submodule of \(L^q(\B,\mu;\B)\).
\end{itemize}
\subsection{Measurable Banach bundles}
According to \cite[Definition 18.1]{AliprantisBorder99}, a given multivalued mapping
\(\boldsymbol\varphi\colon\X\twoheadrightarrow\Y\) between a measurable space \((\X,\Sigma)\)
and a topological space \((\Y,\tau)\) is said to be \textbf{weakly measurable} provided it satisfies
\[
\{x\in\X\;|\;\boldsymbol\varphi(x)\cap U\neq\varnothing\}\in\Sigma\quad\text{ for every }U\in\tau.
\]
The following result, which follows from the Kuratowski--Ryll-Nardzewski selection theorem
(see \cite[Corollary 18.14]{AliprantisBorder99}), gives a useful criterion to detect weakly measurable multivalued mappings:
\begin{proposition}\label{prop:Castaing}
Let \((\X,\Sigma)\) be a measurable space, \(\B\) a separable Banach space, and \({\bf C}\colon\X\twoheadrightarrow\B\)
a multivalued mapping such that \({\bf C}(x)\) is a closed non-empty subset of \(\B\) for every \(x\in\X\). Then \(\bf C\)
is weakly measurable if and only if there exists a sequence \((v_n)_{n\in\N}\) of measurable maps
\(v_n\colon\X\to\B\) such that \({\bf C}(x)={\rm cl}_\B(\{v_n(x)\,:\,n\in\N\})\) for every \(x\in\X\).
\end{proposition}

Let now us recall the notion of measurable Banach \(\B\)-bundle that was introduced in \cite{DMLP21}:
\begin{definition}[Measurable \(\B\)-bundle]\label{def:bundle}
Let \((\X,\Sigma)\) be a measurable space and \(\B\) a Banach space. Then a multivalued mapping
\({\bf E}\colon\X\twoheadrightarrow\B\) is said to be a \textbf{measurable \(\B\)-bundle} on \(\X\) provided it is weakly measurable
and \({\bf E}(x)\) is a closed vector subspace of \(\B\) for every \(x\in\X\).
\end{definition}
Given a finite measure space \((\X,\Sigma,\mu)\), we denote by \(\mathscr D_\mu(\B)\) the set of
measurable \(\B\)-bundles on \(\X\) quotiented up to \(\mu\)-a.e.\ equality. For any \({\bf E}\in\mathscr D_\mu(\B)\)
and \(q\in(1,\infty)\), the \textbf{\(q\)-section space}
\begin{equation}\label{eq:def_sect_bundle}
\Gamma_q({\bf E})\coloneqq\big\{v\in L^q(\X,\mu;\B)\;\big|\;v(x)\in{\bf E}(x)\,\text{ for }\mu\text{-a.e.\ }x\in\X\big\}
\end{equation}
is a closed \(L^\infty(\mu)\)-submodule of \(L^q(\X,\mu;\B)\), thus it is also a Banach space. Notice that \(L^q(\X,\mu;\B)\)
is the \(q\)-section space of the measurable \(\B\)-bundle on \(\X\) whose fibers are identically equal to \(\B\).
\medskip

The space \(\mathscr D_\mu(\B)\) is endowed with a natural partial order \(\preceq\): given any
\({\bf E},{\bf F}\in\mathscr D_\mu(\B)\), we declare that \({\bf E}\preceq{\bf F}\) if and only if
\({\bf E}(x)\subseteq{\bf F}(x)\) for \(\mu\)-a.e.\ \(x\in\X\). When \(\B\) is separable, it holds that
\({\bf E}\mapsto\Gamma_q({\bf E})\) is an order isomorphism between \(\mathscr D_\mu(\B)\) and the set of all closed
\(L^\infty(\mu)\)-submodules of \(L^q(\X,\mu;\B)\) (ordered by inclusion). See the proof of \cite[Proposition 2.22]{LPR21}
after \cite{LP18,DMLP21}.
\medskip

In this paper, the most relevant example of measurable \(\B\)-bundle is the \emph{\(\mu\)-tangent bundle}:
\begin{definition}[\(\mu\)-tangent bundle]\label{def:tg_bundle}
Let \((\B,\mu)\) be a weighted Banach space and let \(q\in(1,\infty)\). Then we define the \textbf{\(q\)-tangent bundle}
of \((\B,\mu)\) as the unique \(T_\mu\B\in\mathscr D_\mu(\B)\) such that
\[
\Gamma_q(T_\mu\B)={\rm cl}_{L^q(\B,\mu;\B)}(D_q(\div_\mu)).
\]
\end{definition}

The above definition is well posed thanks to the fact that the space \({\rm cl}_{L^q(\B,\mu;\B)}(D_q(\div_\mu))\)
is a closed \(L^\infty(\mu)\)-submodule of \(L^q(\B,\mu;\B)\).
\subsubsection*{Dual of a section space}
Let us now recall a characterisation of the dual of the \(q\)-section space of a measurable \(\B\)-bundle
obtained in \cite[Section 3.2]{LPV22}. Let \((\X,\Sigma,\mu)\) be a finite measure space, \(\B\) a separable
Banach space, \({\bf E}\) a measurable \(\B\)-bundle on \(\X\), and \(p\in(1,\infty)\). Then we denote by
\(\Gamma_p({\bf E}^*_{w^*})\) the space of \(\mu\)-a.e.\ equivalence classes of those maps
\(\omega\colon\X\to\bigsqcup_{x\in\X}{\bf E}(x)^*\) that verify the following properties:
\begin{itemize}
\item \(\omega(x)\in{\bf E}(x)^*\) for every \(x\in\X\).
\item The function \(\X\ni x\mapsto\omega(v)(x)\coloneqq\langle\omega(x),v(x)\rangle\) is measurable for every \(v\in\Gamma_q({\bf E})\).
\item The \(\mu\)-a.e.\ equivalence class \(|\omega|_{{\bf E}^*}\) of the function \(\X\ni x\mapsto\|\omega(x)\|_{{\bf E}(x)^*}\) belongs to \(L^p(\mu)\).
\end{itemize}
The space \(\Gamma_p({\bf E}^*_{w^*})\) has a natural structure of module over \(L^\infty(\mu)\) and
it is a Banach space if endowed with the norm \(\omega\mapsto\||\omega|_{{\bf E}^*}\|_{L^p(\mu)}\).
The dual of \(\Gamma_q({\bf E})\) can be identified with \(\Gamma_p({\bf E}^*_{w^*})\):
\begin{theorem}[Dual of a section space]\label{thm:dual_sect_space}
Let \((\X,\Sigma,\mu)\) be a finite measure space and \(q\in(1,\infty)\). Let \(\B\) be a separable Banach space
and \({\bf E}\in\mathscr D_\mu(\B)\). Then it holds that \(\Gamma_q({\bf E})^*\cong\Gamma_p({\bf E}^*_{w^*})\).
\end{theorem}

For the proof of the above result, see \cite[Theorem 3.9]{LPV22}. In the case where \({\bf E}_\B\) is the bundle
whose fibers are constantly equal to \(\B\), we shall write \(L^p_{w^*}(\X,\mu;\B^*)\) instead of \(\Gamma_p(({\bf E}_\B)_{w^*}^*)\).
\begin{definition}\label{def:gen_mod}
Let \((\X,\Sigma,\mu)\) be a finite measure space and let \(q\in(1,\infty)\). Let \(\B\) be a Banach space
and \({\bf E}\in\mathscr D_\mu(\B)\). Then a vector subspace \(\mathcal V\) of \(\Gamma_q({\bf E})\) is said to
\textbf{generate} \(\Gamma_q({\bf E})\) provided
\[
{\rm cl}_{\Gamma_q({\bf E})}\bigg(\bigg\{\sum_{i=1}^n\1_{E_i}v_i\;\bigg|\;n\in\N,\,(E_i)_{i=1}^n\subseteq\Sigma
\text{ partition of }\X,\,(v_i)_{i=1}^n\subseteq\mathcal V\bigg\}\bigg)=\Gamma_q({\bf E}),
\]
or equivalently that the \(L^\infty(\mu)\)-linear span of \(\mathcal V\) is dense in \(\Gamma_q({\bf E})\).
Similarly, we say that a vector subspace \(\mathcal W\) of \(\Gamma_p({\bf E}_{w^*}^*)\) generates
\(\Gamma_p({\bf E}_{w^*}^*)\) if the \(L^\infty(\mu)\)-linear span of \(\mathcal W\) is dense in \(\Gamma_p({\bf E}_{w^*}^*)\).
\end{definition}
\begin{definition}\label{def:restr_Fr_diff}
Let \((\X,\Sigma,\mu)\) be a finite measure space and \(p\in(1,\infty)\). Let \(\B\) be a separable Banach space.
Assume that \({\bf E},{\bf F}\in\mathscr D_\mu(\B)\) satisfy \({\bf E}\preceq{\bf F}\). Then for any given element
\(\omega\in\Gamma_p({\bf F}_{w^*}^*)\) we define \(\omega_{|{\bf E}}\in\Gamma_p({\bf E}_{w^*}^*)\) as
\(\omega_{|{\bf E}}(x)\coloneqq\omega(x)|_{{\bf E}(x)}\in{\bf E}(x)^*\) for \(\mu\)-a.e.\ \(x\in\X\).
\end{definition}

Given a weighted Banach space \((\B,\mu)\), a bundle \({\bf E}\in\mathscr D_\mu(\B)\), and an exponent \(p\in(1,\infty)\),
we can consider the restricted Fr\'{e}chet differential
\begin{equation}\label{eq:restr_diff}
\d_{|{\bf E}}\colon C^1(\B)\cap\LIP_b(\B)\to\Gamma_p({\bf E}_{w^*}^*)\cong\Gamma_q({\bf E})^*,
\end{equation}
which is given by \(\d_{|{\bf E}}f\coloneqq\d f_{|{\bf E}}\in\Gamma_p({\bf E}_{w^*}^*)\) for every \(f\in C^1(\B)\cap\LIP_b(\B)\).
\begin{remark}\label{rmk:equiv_dual_sect}{\rm
In the case where \(\B\) is a separable Asplund space and \({\bf E}\in\mathscr D_\mu(\B)\) is given, the elements of \(\Gamma_p({\bf E}_{w^*}^*)\)
can be represented in a more concrete way, which we are going to describe. Recalling that \(L^q(\X,\mu;\B)^*\cong L^p(\X,\mu;\B^*)\), we can define
\([\,\cdot\,]_{{\bf E}^*}\colon L^p(\X,\mu;\B^*)\to L^q(\mu)\) as
\begin{equation}\label{eq:ptwse_seminorm}
[\omega]_{{\bf E}^*}\coloneqq\bigvee\{\omega(v)\;|\;v\in\Gamma_q({\bf E}),\,|v|_\B\leq 1\}\leq|\omega|_{\B^*}\quad\text{ for every }\omega\in L^p(\X,\mu;\B^*).
\end{equation}
One can readily check that \([\,\cdot\,]_{{\bf E}^*}\) is a `pointwise seminorm' on \(L^p(\X,\mu;\B^*)\), meaning that
\[\begin{split}
[\omega]_{{\bf E}^*}\geq 0&\quad\text{ for every }\omega\in L^p(\X,\mu;\B^*),\\
[\omega+\eta]_{{\bf E}}\leq[\omega]_{{\bf E}}+[\eta]_{{\bf E}}&\quad\text{ for every }\omega,\eta\in L^p(\X,\mu;\B^*),\\
[f\omega]_{{\bf E}^*}=|f|[\omega]_{{\bf E}^*}&\quad\text{ for every }f\in L^\infty(\mu)\text{ and }\omega\in L^p(\X,\mu;\B^*).
\end{split}\]
Moreover, the `restriction' operator \(\pi_{\bf E}\colon L^p(\X,\mu;\B^*)\to\Gamma_p({\bf E}_{w^*}^*)\), which we define as
\begin{equation}\label{eq:restr_op}
\pi_{\bf E}(\omega)\coloneqq\omega_{|{\bf E}}\in\Gamma_p({\bf E}_{w^*}^*)\quad\text{ for every }\omega\in L^p(\X,\mu;\B^*),
\end{equation}
is \(L^\infty(\mu)\)-linear and satisfies \(|\pi_{\bf E}(\omega)|_{{\bf E}^*}=[\omega]_{{\bf E}^*}\) for every \(\omega\in L^p(\X,\mu;\B^*)\).
Finally, by applying the Hahn--Banach theorem, one can prove that \(\pi_{\bf E}\) is surjective and that for any \(\tilde\omega\in\Gamma_p({\bf E}_{w^*}^*)\)
there exists a (non-unique) \(\omega\in L^p(\X,\mu;\B^*)\) such that \(\pi_{\bf E}(\omega)=\tilde\omega\) and \(|\omega|_{\B^*}=|\tilde\omega|_{{\bf E}^*}\).
\fr}\end{remark}
\begin{lemma}\label{lem:technical_mod}
Let \((\X,\Sigma,\mu)\) be a finite measure space and \(p\in(1,\infty)\). Let \(\B\) be a separable Banach space
and \({\bf E},{\bf F}\in\mathscr D_\mu(\B)\). Assume that \(\mathcal W\) is a generating vector subspace of
\(L^p_{w^*}(\X,\mu;\B^*)\) such that \(|\omega_{|{\bf E}}|_{{\bf E}^*}\leq|\omega_{|{\bf F}}|_{{\bf F}^*}\)
for every \(\omega\in\mathcal W\). Then it holds that \({\bf E}\preceq{\bf F}\).
\end{lemma}
\begin{proof}
We argue by contradiction: assume there exists \(P\in\Sigma\) with \(\mu(P)>0\) and
\({\bf E}(x)\setminus{\bf F}(x)\neq\varnothing\) for \(\mu\)-a.e.\ \(x\in P\). Applying Proposition
\ref{prop:Castaing}, we can find \(v\in\Gamma_q({\bf E})\) and \(\lambda>0\) for which (up to shrinking \(P\))
it holds that \(v(x)\notin{\bf F}(x)\) and \(\|v(x)\|_\B\geq\lambda\) for \(\mu\)-a.e.\ \(x\in P\). By the Hahn--Banach theorem,
there exists \(\eta\in L^p_{w^*}(\X,\mu;\B^*)\) such that \(\eta(v)=|v|_\B^q\) holds \(\mu\)-a.e.\ on \(P\) and
\(\eta(u)=0\) for every \(u\in\Gamma_q({\bf F})\). Fix any \(\varepsilon\in(0,\lambda^{q-1}/2)\).
Since \(\mathcal W\) generates \(L^p_{w^*}(\X,\mu;\B^*)\), we can find \(\omega\in\mathcal W\) such that
(up to shrinking \(P\) further) we have that \(|\eta-\omega|_{\B^*}\leq\varepsilon\) holds \(\mu\)-a.e.\ on \(P\).
In particular, we have that \(\omega(u)\leq\eta(u)+|\omega-\eta|_{\B^*}|u|_\B\leq\varepsilon|u|_\B\) holds \(\mu\)-a.e.\ on
\(P\) for every given \(u\in\Gamma_q({\bf F})\), whence it follows that \(|\omega_{|{\bf F}}|_{{\bf F}^*}\leq\varepsilon\) holds
\(\mu\)-a.e.\ on \(P\). Hence, we have (again, \(\mu\)-a.e.\ on \(P\)) that
\[
\lambda^{q-1}|v|_\B\leq|v|_\B^q=\eta(v)\leq|\omega(v)|+\varepsilon|v|_\B\leq(|\omega_{|{\bf E}}|_{{\bf E}^*}+\varepsilon)|v|_\B
\leq(|\omega_{|{\bf F}}|_{{\bf F}^*}+\varepsilon)|v|_\B\leq 2\varepsilon|v|_\B,
\]
so that \(\lambda^{q-1}\leq 2\varepsilon\), which contradicts our choice of \(\varepsilon\). Therefore, we conclude that \({\bf E}\preceq{\bf F}\).
\end{proof}

We now present another technical result, which will be useful in the later sections:
\begin{lemma}\label{lem:cyl_generators_cotg}
Let \((\B,\mu)\) be a weighted Banach space and let \(p\in(1,\infty)\). Let \(\mathscr A\) be a compatible subalgebra of \(C^1(\B)\cap\LIP_b(\B)\) containing \(\B^*\).
Then \(\{\d f\,:\,f\in\mathscr A\}\) generates \(L^p(\B,\mu;\B^*)\). If in addition the space \(\B\) is Asplund, then \(\{\d_{|{\bf E}}f\,:\,f\in\mathscr A\}\)
generates \(\Gamma_p({\bf E}_{w^*}^*)\) for every \({\bf E}\in\mathscr D_\mu(\B)\).
\end{lemma}
\begin{proof}
By the very definition of \(L^p(\B,\mu;\B^*)\), for any \(\omega\in L^p(\B,\mu;\B^*)\) and \(\varepsilon>0\) we can find a simple map \(\eta=\sum_{i=1}^n\1_{E_i}\omega_i\),
where \((E_i)_{i=1}^n\) is a Borel partition of \(\B\) and \(\omega_1,\ldots,\omega_n\in\B^*\), such that \(\|\omega-\eta\|_{L^p(\B,\mu;\B^*)}\leq\varepsilon\).
Since \(\B^*\subseteq\mathscr A\) by assumption and the Fr\'{e}chet differential of an element of \(\B^*\) coincides with the element itself, the above shows that \(\{\d f\,:\,f\in\mathscr A\}\)
generates \(L^p(\B,\mu;\B^*)\).

Let us now assume further that \(\B\) is Asplund. Fix any \({\bf E}\in\mathscr D_\mu(\B)\). Given any \(\omega\in\Gamma_p({\bf E}_{w^*}^*)\), we can find
\(\bar\omega\in L^p(\B,\mu;\B^*)\) with \(\bar\omega_{|{\bf E}}=\omega\) (and \(|\omega|_{\B^*}=|\omega|_{{\bf E}^*}\)), see Remark \ref{rmk:equiv_dual_sect}.
Now fix \(\varepsilon>0\). The first part of the statement yields the existence of functions \(f_1,\ldots,f_n\in\mathscr A\) and of a Borel partition \(E_1,\ldots,E_n\) of \(\B\)
such that \(\bar\eta\coloneqq\sum_{i=1}^n\1_{E_i}\d f_i\in L^p(\B,\mu;\B^*)\) satisfies \(\|\bar\omega-\bar\eta\|_{L^p(\B,\mu;\B^*)}\leq\varepsilon\). Letting
\(\eta\coloneqq\bar\eta_{|{\bf E}}=\sum_{i=1}^n\1_{E_i}\d_{|{\bf E}}f_i\in\Gamma_p({\bf E}_{w^*}^*)\), we have that
\(\|\omega-\eta\|_{\Gamma_p({\bf E}_{w^*}^*)}\leq\|\bar\omega-\bar\eta\|_{L^p(\B,\mu;\B^*)}\leq\varepsilon\). This shows that \(\{\d_{|{\bf E}}f\,:\,f\in\mathscr A\}\)
generates \(\Gamma_p({\bf E}_{w^*}^*)\), thus completing the proof of the statement.
\end{proof}

We point out that the algebra of smooth cylindrical functions on \(\B\) contains \(\B^*\). In particular, Lemma \ref{lem:cyl_generators_cotg} can be applied to every
(compatible) subalgebra of \(C^1(\B)\cap\LIP_b(\B)\) that contains all smooth cylindrical functions.
\subsubsection*{Pullback bundle}
Let \((\X,\sfd_\X,\mu_\X)\) and \((\Y,\sfd_\Y,\mu_\Y)\) be complete and separable metric spaces endowed with finite Borel
measures. Let \(\varphi\colon\X\to\Y\) be a Borel map satisfying \(\varphi_\#\mu_\X\leq C\mu_\Y\) for some constant \(C>0\). Let \(\B\)
be a separable Banach space and let \({\bf E}\in\mathscr D_{\mu_\Y}(\B)\). Following \cite[Definition C.1]{DMLP21}, we then
define the \textbf{pullback bundle} \(\varphi^*{\bf E}\in\mathscr D_{\mu_\X}(\B)\) of \(\bf E\) as
\begin{equation}\label{eq:def_pullback_bundle}
(\varphi^*{\bf E})(x)\coloneqq{\bf E}(\varphi(x))\quad\text{ for }\mu_\X\text{-a.e.\ }x\in\X.
\end{equation}
For any exponent \(q\in(1,\infty)\), the pullback operator that we considered in \eqref{eq:pullback_Leb-Boch} induces a pullback
operator \(\varphi^*\colon\Gamma_q({\bf E})\to\Gamma_q(\varphi^*{\bf E})\). Furthermore, if \(\mathcal V\) is a generating
vector subspace of \(\Gamma_q({\bf E})\), then
\begin{equation}\label{eq:gen_pullback}
\{\varphi^*v\;|\;v\in\mathcal V\}\quad\text{ generates }\Gamma_q(\varphi^*{\bf E}).
\end{equation}
For a proof of the above claim, see \cite[Theorem C.3]{DMLP21}, taking also \cite[Remark C.4]{DMLP21} into account.
\section{Main results}
First, we provide several characterisations of \(W^{1,p}(\B,\mu)\) on a weighted (reflexive) Banach space:
\begin{theorem}[Identification of the Sobolev space]\label{thm:main_Sob}
Let \((\B,\mu)\) be a weighted Banach space such that \(\B\) has the Radon--Nikod\'{y}m property and let \(p\in(1,\infty)\).
Let \(f\in L^p(\mu)\) be a given function. Then the following conditions are equivalent:
\begin{itemize}
\item[\(\rm a)\)] It holds that \(f\in W^{1,p}(\B,\mu)\).
\item[\(\rm b)\)] There exists a (necessarily unique) element \(L_f\in D_q(\div_\mu)^*\cong\Gamma_q(T_\mu\B)^*\) such that
\begin{equation}\label{eq:char_L_f_cl1}
\int L_f(v)\,\d\mu=-\int f\,\div_\mu(v)\,\d\mu\quad\text{ for every }v\in D_q(\div_\mu).
\end{equation}
\end{itemize}
If in addition the space \(\B\) is reflexive, then {\rm a)} and {\rm b)} are also equivalent to the following condition:
\begin{itemize}
\item[\(\rm c)\)] Given any compatible subalgebra \(\mathscr A\) of \(C^1(\B)\cap\LIP_b(\B)\), it holds that \({\rm WD}_\mu^p(f;\mathscr A)\neq\varnothing\),
where the family \({\rm WD}_\mu^p(f;\mathscr A)\) of all \textbf{\((\mathscr A,p)\)-weak differentials} of \(f\) is given by
\[
{\rm WD}_\mu^p(f;\mathscr A)\coloneqq\Big\{\omega\in L^p(\B,\mu;\B^*)\;\Big|\;f_n\rightharpoonup f\text{ and }\d f_n\rightharpoonup\omega\text{ weakly, for some }(f_n)_n\subseteq\mathscr A\Big\}.
\]
\end{itemize}
\end{theorem}

At the same time, we can also obtain various characterisations of the minimal \(p\)-weak upper gradient \(|\D_\mu f|\) of any Sobolev function \(f\in W^{1,p}(\B,\mu)\):
\begin{theorem}[Identification of the minimal weak upper gradient]\label{thm:main_mwug}
Let \((\B,\mu)\) be a weighted Banach space such that \(\B\) has the Radon--Nikod\'{y}m property. Let \(p\in(1,\infty)\) be given. Then
\begin{equation}\label{eq:def_d_mu}
\d_\mu f\coloneqq L_f\in\Gamma_q(T_\mu\B)^*\quad\text{ for every }f\in W^{1,p}(\B,\mu)
\end{equation}
defines a linear map \(\d_\mu\colon W^{1,p}(\B,\mu)\to\Gamma_q(T_\mu\B)^*\), which we call the \textbf{\(p\)-weak differential}, with
\begin{equation}\label{eq:char_L_f_cl2}
|\D_\mu f|=|\d_\mu f|_{(T_\mu\B)^*}\quad\text{ for every }f\in W^{1,p}(\B,\mu).
\end{equation}
Moreover, we have that \(\d_\mu f=\d_{|T_\mu\B}f\) holds for every \(f\in C^1(\B)\cap\LIP_b(\B)\), thus in particular
\begin{equation}\label{eq:char_L_f_cl3}
|\D_\mu f|=\bigvee_{v\in D_q(\div_\mu)}\1_{\{v\neq 0\}}\frac{\d f(v)}{|v|_\B}\quad\text{ for every }f\in C^1(\B)\cap\LIP_b(\B).
\end{equation}
If in addition \(\B\) is reflexive, for any compatible subalgebra \(\mathscr A\) of \(C^1(\B)\cap\LIP_b(\B)\) it holds that
\begin{equation}\label{eq:id_with_WD}
|\D_\mu f|=|\omega|_{\B^*}\quad\text{ for every }f\in W^{1,p}(\B,\mu)\text{ and }\omega\in{\rm WD}_\mu^p(f;\mathscr A).
\end{equation}
\end{theorem}

Theorems \ref{thm:main_Sob} and \ref{thm:main_mwug} will be proved in Section \ref{s:soboident}. Furthermore, we show that the \(\mu\)-tangent bundle \(T_\mu\B\)
can be alternatively defined in terms of the velocity of curves selected by \(q\)-test plans:
\begin{theorem}[Identification of the tangent bundle]\label{thm:main_tg_bundle}
Let \((\B,\mu)\) be a weighted Banach space such that \(\B\) is an Asplund space that has the Radon--Nikod\'{y}m property. Let \(q\in(1,\infty)\) be given. Then
the \(\mu\)-tangent bundle \(T_\mu\B\) can be equivalently characterised as follows:
\begin{itemize}
\item[\(\rm i)\)] \(T_\mu\B\) is the unique minimal element of \((\mathscr D_\mu(\B),\preceq)\) such that
\begin{equation}\label{eq:tg_bundle_span_tp_claim}
\dot\gamma_t\in T_\mu\B(\gamma_t)\quad\text{ for every }\ppi\in\Pi_q(\B,\mu)\text{ and }\hat\ppi\text{-a.e.\ }(\gamma,t)\in\mathfrak C(\B).
\end{equation}
\item[\(\rm ii)\)] If \(\ppi\) is a given master \(q\)-test plan on \((\B,\mu)\), then \(T_\mu\B(x)=\{0_\B\}\) for \(\mu\)-a.e.\ \(x\in\B\setminus{\rm S}_\sppi\) and
\[
T_\mu\B(x)={\rm cl}_\B\bigg({\rm span}\bigg\{v\in\B\;\bigg|\;\hat\ppi_x\Big(\Big\{(\gamma,t)\in\mathfrak C(\B)\,:\,\|\dot\gamma_t-v\|_\B<\varepsilon\Big\}\Big)>0\,\text{ for all }\varepsilon>0\bigg\}\bigg)
\]
for \(\e_\#\hat\ppi\)-a.e.\ \(x\in\B\), where \(\hat\ppi=\int\hat\ppi_x\,\d(\e_\#\hat\ppi)(x)\) denotes the disintegration of \(\hat\ppi\) along \(\e\).
\end{itemize}
\end{theorem}

See Section \ref{ss:id_tg_bundle} for the proof of Theorem \ref{thm:main_tg_bundle}.
Combining Theorems \ref{thm:main_Sob} and \ref{thm:main_mwug} with well-known facts in the Banach space theory, we obtain the following useful corollary:
\begin{corollary}\label{cor:refl_Sobolev}
Let \((\B,\mu)\) be a weighted Banach space and \(p\in(1,\infty)\). Then the following hold:
\begin{itemize}
\item[\(\rm i)\)] If \(\B\) is reflexive, then \(W^{1,p}(\B,\mu)\) is reflexive.
\item[\(\rm ii)\)] If the dual \(\B^*\) is uniformly convex, then \(W^{1,p}(\B,\mu)\) is uniformly convex.
\item[\(\rm iii)\)] If \(\B\) is a Hilbert space and \(p=2\), then \(W^{1,2}(\B,\mu)\) is a Hilbert space.
\end{itemize}
In particular, if the space \(\B\) is reflexive, then every compatible subalgebra \(\mathscr A\) of \(C^1(\B)\cap\LIP_b(\B)\) is strongly dense in \(W^{1,p}(\B,\mu)\),
and the \(p\)-weak differential operator \(\d_\mu\colon W^{1,p}(\B,\mu)\to\Gamma_q(T_\mu\B)^*\) is the unique linear and continuous extension of the map
\(\d_{|T_\mu\B}\colon\mathscr A\to\Gamma_q(T_\mu\B)^*\).
\end{corollary}
\begin{proof}
Endow \(L^p(\mu)\times D_q(\div_\mu)^*\) with the norm \(\|(f,L)\|\coloneqq(\|f\|_{L^p(\mu)}^p+\|L\|_{D_q(\div_\mu)^*}^p)^{1/p}\). Then
\begin{equation}\label{eq:reflex_Sob}
W^{1,p}(\B,\mu)\ni f\mapsto(f,L_f)\in L^p(\mu)\times D_q(\div_\mu)^*\quad\text{ is a linear isometry,}
\end{equation}
thanks to the linearity of the map \(f\mapsto L_f\) and to \eqref{eq:char_L_f_cl2}. Let us now distinguish the three cases:
\smallskip

\noindent\(\boldsymbol{\rm i)}\) By well-known stability properties of reflexivity, we have the following chain of implications:
\[\begin{split}
\B\text{ reflexive}&\quad\Longrightarrow\quad L^q(\B,\mu;\B)\text{ reflexive}\quad\Longrightarrow\quad L^p(\B,\mu;\B^*)\text{ reflexive}\\
&\quad\Longrightarrow\quad D_q(\div_\mu)^*\text{ reflexive}\quad\Longrightarrow\quad L^p(\mu)\times D_q(\div_\mu)^*\text{ reflexive}.
\end{split}\]
Indeed, the Lebesgue--Bochner space \(L^q(\B,\mu;\B)\) is reflexive if (and only if) \(\B\) is reflexive, we have that \(L^q(\B,\mu;\B)^*\cong L^p(\B,\mu;\B^*)\) when \(\B\) is reflexive,
and \(D_q(\div_\mu)^*\) is isometrically isomorphic to a quotient of \(L^p(\B,\mu;\B^*)\). Taking also \eqref{eq:reflex_Sob} into account, the implication in i) is then proved.
\smallskip

\noindent\(\boldsymbol{\rm ii)}\) Suppose \(\B^*\) is uniformly convex. Then \(L^p(\B,\mu;\B^*)\) is uniformly convex. Since uniformly convex spaces are reflexive, the previous discussion
ensures that \(D_q(\div_\mu)^*\) is isometrically isomorphic to a quotient of \(L^p(\B,\mu;\B^*)\). It follows that \(D_q(\div_\mu)^*\) is uniformly convex, so that
\(L^p(\mu)\times D_q(\div_\mu)^*\) is uniformly convex. Recalling \eqref{eq:reflex_Sob} again, we conclude that \(W^{1,p}(\B,\mu)\) is uniformly convex.
\smallskip

\noindent\(\boldsymbol{\rm iii)}\) Suppose \(\B\) is Hilbert and \(p=2\). Since \(D_2(\div_\mu)^*\) is a quotient of the Hilbert space \(L^2(\B,\mu;\B^*)\),
we deduce that \(D_2(\div_\mu)^*\) is Hilbert. Thanks to \eqref{eq:reflex_Sob}, it follows that \(W^{1,2}(\B,\mu)\) is Hilbert.
\smallskip

Finally, assume that \(\B\) is reflexive, so that \(W^{1,p}(\B,\mu)\) is reflexive by i) and thus every compatible subalgebra \(\mathscr A\) of \(C^1(\B)\cap\LIP_b(\B)\)
is strongly dense in \(W^{1,p}(\B,\mu)\). In particular, the linear continuous operator \(\d_{|T_\mu\B}\colon\mathscr A\to\Gamma_q(T_\mu\B)^*\) can be uniquely extended
to a linear continuous operator \(T\colon W^{1,p}(\B,\mu)\to\Gamma_q(T_\mu\B)^*\). Given that \(\d_\mu\) is a linear continuous extension of \(\d_{|T_\mu\B}\) by
Theorem \ref{thm:main_mwug}, we conclude that \(T=\d_\mu\). This proves the last part of the statement.
\end{proof}
\begin{remark}{\rm
The statements of Corollary \ref{cor:refl_Sobolev} were already known:\\
\(\boldsymbol{\rm i)}\) was proved in \cite[Corollary 5.3.11]{Sav19}.\\
\(\boldsymbol{\rm ii)}\) follows from the results of \cite{Sodini22}.\\
\(\boldsymbol{\rm iii)}\) was proved in \cite{DMGPS20} (as Hilbert spaces are \({\sf CAT}(0)\) spaces); see also \cite[Corollary 5.3.11]{Sav19}.
\fr}\end{remark}
\subsection{Proof of Theorems \ref{thm:main_Sob} and \ref{thm:main_mwug}}\label{s:soboident}
First, we show that each test plan on a weighted Banach space with the Radon--Nikod\'{y}m property induces a vector field.
The proof is inspired by \cite[Proposition 2.4]{DiMarino14} and \cite[Proposition 7.2.3]{DiMarinoPhD}.
\begin{lemma}[Vector field induced by a test plan]\label{lem:vf_induced_tp}
Let \((\B,\mu)\) be a weighted Banach space such that \(\B\) has the Radon--Nikod\'{y}m property. Let \(q\in(1,\infty)\) and \(\ppi\in\Pi_q(\B,\mu)\) be given. Let us define
\begin{equation}\label{eq:def_vf_induced_tp}
v_\sppi(x)\coloneqq\frac{\d(\e_\#\hat\ppi)}{\d\mu}(x)\int\dot\gamma_t\,\d\hat\ppi_x(\gamma,t)\in\B\quad\text{ for }\mu\text{-a.e.\ }x\in\B.
\end{equation}
Then it holds that \(v_\sppi\in L^q(\B,\mu;\B)\) and
\begin{equation}\label{eq:vf_induced_tp_cl1}
\int h|v_\sppi|_\B\,\d\mu\leq\int\!\!\!\int_0^1 h(\gamma_t)\|\dot\gamma_t\|_\B\,\d t\,\d\ppi(\gamma)\quad\text{ for every }h\colon\B\to[0,+\infty)\text{ Borel.}
\end{equation}
Moreover, it holds that \(v_\sppi\in D_q(\div_\mu)\) and
\begin{equation}\label{eq:vf_induced_tp_cl2}
\div_\mu(v_\sppi)=\frac{\d(\e_0)_\#\ppi}{\d\mu}-\frac{\d(\e_1)_\#\ppi}{\d\mu}.
\end{equation}
\end{lemma}
\begin{proof}
First of all, the map \(v_\sppi\colon\B\to\B\) is \(\mu\)-measurable thanks to the measurability of \(x\mapsto\hat\ppi_x\). Moreover, given any Borel function \(h\colon\B\to[0,+\infty)\), we have that
\[
\int h|v_\sppi|_\B\,\d\mu\leq\int h(x)\frac{\d(\e_\#\hat\ppi)}{\d\mu}(x)\int\|\dot\gamma_t\|_\B\,\d\hat\ppi_x(\gamma,t)\,\d\mu(x)
=\int\!\!\!\int_0^1 h(\gamma_t)\|\dot\gamma_t\|_\B\,\d t\,\d\ppi(\gamma),
\]
which gives \eqref{eq:vf_induced_tp_cl1}, thus in particular \(v_\sppi\in L^q(\B,\mu;\B)\) by H\"{o}lder's inequality. Finally, for any given function
\(f\in C^1(\B)\cap\LIP_b(\B)\) we can compute
\[\begin{split}
\int\d f(v_\sppi)\,\d\mu&=\int\!\!\!\int\langle\d_{\gamma_t}f,\dot\gamma_t\rangle\,\d\hat\ppi_x(\gamma,t)\,\d(\e_\#\hat\ppi)(x)
=\int\!\!\!\int_0^1(f\circ\gamma)'_t\,\d t\,\d\ppi(\gamma)\\
&=\int f(\gamma_1)-f(\gamma_0)\,\d\ppi(\gamma)=\int f\bigg(\frac{\d(\e_1)_\#\ppi}{\d\mu}-\frac{\d(\e_0)_\#\ppi}{\d\mu}\bigg)\,\d\mu,
\end{split}\]
which shows that \(v_\sppi\in D_q(\div_\mu)\) and that the identity in \eqref{eq:vf_induced_tp_cl2} holds. The proof is complete.
\end{proof}
\begin{proof}[Proof of Theorems \ref{thm:main_Sob} and \ref{thm:main_mwug}]
\ 
\smallskip

{\bf Step 1: proof of \({\rm a)}\Rightarrow{\rm b)}\) and \(|\d_\mu f|_{(T_\mu\B)^*}\leq|\D_\mu f|\).}\\
Assume that \(f\in W^{1,p}(\B,\mu)\). Pick any sequence \((f_n)_n\subseteq C^1(\B)\cap\LIP_b(\B)\) such that \(f_n\to f\) and \(|\d f_n|_{\B^*}\to|\D_\mu f|\) in \(L^p(\mu)\).
Then we have that
\begin{equation}\label{eq:char_Sob_L_f_aux1}
\bigg|\int f\,\div_\mu(v)\,\d\mu\bigg|=\lim_{n\to\infty}\bigg|\int f_n\,\div_\mu(v)\,\d\mu\bigg|
\leq\lim_{n\to\infty}\int|\d f_n|_{\B^*}|v|_\B\,\d\mu=\int|\D_\mu f||v|_\B\,\d\mu
\end{equation}
for every \(v\in D_q(\div_\mu)\). Letting \(\mu_v\coloneqq|\D_\mu f||v|_\B\mu\), we also define \(T_v\colon C^1(\B)\cap\LIP_b(\B)\to\R\) as
\[
T_v(h)\coloneqq-\int f\,\div_\mu(hv)\,\d\mu\quad\text{ for every }h\in C^1(\B)\cap\LIP_b(\B).
\]
By \eqref{eq:Leibniz_div_mu}, the map \(T_v\) is linear. Also, \eqref{eq:char_Sob_L_f_aux1} gives \(|T_v(h)|\leq\|h\|_{L^1(\mu_v)}\) for all
\(h\in C^1(\B)\cap\LIP_b(\B)\). Since \(C^1(\B)\cap\LIP_b(\B)\) is dense in \(L^1(\mu_v)\) and the dual of \(L^1(\mu_v)\) is \(L^\infty(\mu_v)\),
there exists a unique function \(\theta_v\in L^\infty(\mu_v)\) such that \(|\theta_v|\leq 1\) and \(T_v(h)=\int h\,\theta_v\,\d\mu_v\)
for every \(h\in C^1(\B)\cap\LIP_b(\B)\). Therefore, letting \(L_f(v)\coloneqq\theta_v|\D_\mu f||v|_\B\in L^1(\mu)\), we conclude that \(|L_f(v)|\leq|\D_\mu f||v|_\B\) and
\begin{equation}\label{eq:char_Lf}
\int h\,L_f(v)\,\d\mu=-\int f\,\div_\mu(hv)\,\d\mu\quad\text{ for every }v\in D_q(\div_\mu)\text{ and }h\in C^1(\B)\cap\LIP_b(\B).
\end{equation}
Choosing \(h\coloneqq 1\), we obtain that \(L_f\) satisfies \eqref{eq:char_L_f_cl1}, while the weak\(^*\) density of \(C^1(\B)\cap\LIP_b(\B)\)
in \(L^\infty(\mu)\) ensures that \(L_f\) is uniquely determined. Since \(L_f\) is linear (thanks to \eqref{eq:char_Lf} and to the linearity of \(\div_\mu\)),
we conclude that \(L_f\in D_q(\div_\mu)^*\) and \(|\d_\mu f|_{(T_\mu\B)^*}=|L_f|_{(T_\mu\B)^*}\leq|\D_\mu f|\), proving the inequality \(\geq\) in \eqref{eq:char_L_f_cl2}.
The linearity of \(\d_\mu\colon W^{1,p}(\B,\mu)\to D_q(\div_\mu)^*\) follows from \eqref{eq:Leibniz_div_mu}.
\smallskip

{\bf Step 2: proof of \({\rm b)}\Rightarrow{\rm a)}\) and \(|\D_\mu f|\leq|\d_\mu f|_{(T_\mu\B)^*}\).}\\
Assume that there exists an element \(L_f\in D_q(\div_\mu)^*\) satisfying \eqref{eq:char_L_f_cl1}. Fix any \(\ppi\in\Pi_q(\B,\mu)\)
and consider the vector field \(v_\sppi\in L^q(\B,\mu;\B)\) induced by \(\ppi\) as in \eqref{eq:def_vf_induced_tp}. Then we can estimate
\[\begin{split}
\int f(\gamma_1)-f(\gamma_0)\,\d\ppi(\gamma)&=\int f\,\d(\e_1)_\#\ppi-\int f\,\d(\e_0)_\#\ppi\overset{\eqref{eq:vf_induced_tp_cl2}}=
-\int f\,\div_\mu(v_\sppi)\,\d\mu=\int L_f(v_\sppi)\,\d\mu\\
&\leq\int|L_f|_{(T_\mu\B)^*}|v_\sppi|_\B\,\d\mu\overset{\eqref{eq:vf_induced_tp_cl1}}\leq\int\!\!\!\int_0^1|L_f|_{(T_\mu\B)^*}(\gamma_t)\|\dot\gamma_t\|_\B\,\d t\,\d\ppi(\gamma).
\end{split}\]
Hence, \(f\in W^{1,p}(\B,\mu)\) and \(|\D_\mu f|\leq|L_f|_{(T_\mu\B)^*}=|\d_\mu f|_{(T_\mu\B)^*}\), thus the proof of \eqref{eq:char_L_f_cl2} is complete.
\smallskip

{\bf Step 3: proof of \eqref{eq:char_L_f_cl3}.}\\
Given that \(\int h\,L_f(v)\,\d\mu=-\int f\,\div_\mu(hv)\,\d\mu=\int h\,\d f(v)\,\d\mu\) holds for every \(f,h\in C^1(\B)\cap\LIP_b(\B)\) and \(v\in D_q(\div_\mu)\), we have that
\(L_f(v)=\d f(v)\) for every \(f\in C^1(\B)\cap\LIP_b(\B)\) and \(v\in D_q(\div_\mu)\), so that \(\d_\mu f=L_f=\d_{|T_\mu\B}f\) for every \(f\in C^1(\B)\cap\LIP_b(\B)\).
In particular, by taking also the identity \eqref{eq:char_L_f_cl2} into account, for any \(f\in C^1(\B)\cap\LIP_b(\B)\) we obtain that
\[
|\D_\mu f|=|\d_{|T_\mu\B}f|_{(T_\mu\B)^*}=\bigvee_{v\in D_q(\div_\mu)}\1_{\{v\neq 0\}}\frac{\d f(v)}{|v|_\B},
\]
thus proving the validity of \eqref{eq:char_L_f_cl3}.
\smallskip

{\bf Step 4: proof of \({\rm a)}\Leftrightarrow{\rm c)}\) and \eqref{eq:id_with_WD}.}\\
Let us now assume in addition that \(\B\) is reflexive. Let \(f\in W^{1,p}(\B,\mu)\) be given. Pick a sequence \((f_n)_n\subseteq\mathscr A\) such that \(f_n\to f\) and
\(|\d f_n|_{\B^*}\to|\D_\mu f|\) in \(L^p(\mu)\). Since \(L^p(\B,\mu;\B^*)\) is reflexive, up to a non-relabelled subsequence we have that \(\d f_n\rightharpoonup\omega\)
weakly in \(L^p(\B,\mu;\B^*)\) for some \(\omega\in L^p(\B,\mu;\B^*)\), so that \(\omega\in{\rm WD}^p_\mu(f;\mathscr A)\). Notice also that we have that \(|\omega|_{\B^*}\leq|\D_\mu f|\).
Conversely, let \(f\in L^p(\mu)\) with \({\rm WD}^p_\mu(f;\mathscr A)\neq\varnothing\) be given. Take any \(\omega\in{\rm WD}^p_\mu(f;\mathscr A)\) and \((f_n)_n\subseteq\mathscr A\)
such that \(f_n\rightharpoonup f\) weakly in \(L^p(\mu)\) and \(\d f_n\rightharpoonup\omega\) weakly in \(L^p(\B,\mu;\B^*)\). Thanks to Mazur's lemma, we can find a sequence
\((g_n)_n\) of convex combinations of \((f_n)_n\) so that \(g_n\to f\) strongly in \(L^p(\mu)\) and \(\d g_n\to\omega\) strongly in \(L^p(\B,\mu;\B^*)\). In particular, it holds
that \(|\d g_n|_{\B^*}\to|\omega|_{\B^*}\) strongly in \(L^p(\mu)\), which implies that \(f\in W^{1,p}(\B,\mu)\) and \(|\D_\mu f|\leq|\omega|_{\B^*}\). All in all,
the property \eqref{eq:id_with_WD} is proved.
\end{proof}

Notice that the implication \({\rm WD}_\mu^p(f;\mathscr A)\neq\varnothing\Rightarrow f\in W^{1,p}(\B,\mu)\) holds for every \(\B\) separable.
\subsection{Proof of Theorem \ref{thm:main_tg_bundle}}\label{ss:id_tg_bundle}
Before passing to the verification of Theorem \ref{thm:main_tg_bundle}, we prove an auxiliary result:
\begin{lemma}\label{lem:char_master_tp_aux}
Let \((\B,\mu)\) be a weighted Banach space such that \(\B\) has the Radon--Nikod\'{y}m property. Let \(q\in(1,\infty)\) and \(\ppi\in\Pi_q(\B,\mu)\) be given.
Let us define the multivalued mapping \(V_\sppi\B\colon\B\twoheadrightarrow\B\) as \(V_\sppi\B(x)\coloneqq{\rm cl}_\B({\rm span}\,S_\sppi\B(x))\) for \(\mu\)-a.e.\ \(x\in\B\), where we set
\begin{equation}\label{eq:char_master_tp_aux}
S_\sppi\B(x)\coloneqq\left\{\begin{array}{ll}
{\rm spt}(\Der_\#\hat\ppi_x)\\
\{0_\B\}
\end{array}\quad\begin{array}{ll}
\text{ for }\e_\#\hat\ppi\text{-a.e.\ }x\in\B,\\
\text{ for }\mu\text{-a.e.\ }x\in\B\setminus{\rm S}_\sppi.
\end{array}\right.
\end{equation}
Then it holds that \(S_\sppi\B\colon\B\twoheadrightarrow\B\) is weakly measurable and \(V_\sppi\B\in\mathscr D_\mu(\B)\).
\end{lemma}
\begin{proof}
First, notice that \(V_\sppi\B(x)\) is a closed vector subspace of \(\B\) for \(\mu\)-a.e.\ \(x\in\B\). Moreover,
\[
\{x\in\B\;|\;S_\sppi\B(x)\cap U\neq\varnothing\}=\{x\in\B\;|\;\hat\ppi_x(\Der^{-1}(U))>0\}\quad\text{ for every }U\subseteq\B\text{ open,}
\]
thus the measurability of \(x\mapsto\hat\ppi_x\) ensures that the multivalued mapping \(S_\sppi\B\) is weakly measurable,
whence it follows that \({\rm cl}_\B(S_\sppi\B)\) is weakly measurable. Thanks to Proposition \ref{prop:Castaing},
we can find a sequence \((v_k)_{k\in\N}\) of Borel maps \(v_k\colon\B\to\B\) such that
\({\rm cl}_\B(\{v_k(x)\,:\,k\in\N\})={\rm cl}_\B(S_\sppi\B(x))\) holds for \(\mu\)-a.e.\ \(x\in\B\). Since for \(\mu\)-a.e.\ \(x\in\B\) we have that
\[
V_\sppi\B(x)={\rm cl}_\B\bigg(\bigg\{\sum_{i=1}^n q_i\,v_{k_i}(x)\;\bigg|\;n\in\N,\,q_1,\ldots,q_n\in\mathbb Q,
\,k_1,\ldots,k_n\in\N\bigg\}\bigg),
\]
we deduce from Proposition \ref{prop:Castaing} that \(V_\sppi\B\) is weakly measurable, thus \(V_\sppi\B\in\mathscr D_\mu(\B)\).
\end{proof}
\begin{proof}[Proof of Theorem \ref{thm:main_tg_bundle}]
\ 
\smallskip

{\bf Step 1: proof of Theorem \ref{thm:main_tg_bundle} i).}\\
Given any \(\ppi\in\Pi_q(\B,\mu)\), we denote by \(\Der_\sppi\in L^q(\mathfrak C(\B),\hat\ppi;\B)\) the equivalence
class of the mapping \(\Der\colon\mathfrak C(\B)\to\B\). We define \({\bf E}_\sppi\in\mathscr D_{\hat\sppi}(\B)\) as
\({\bf E}_\sppi(\gamma,t)\coloneqq\R\,{\rm Der}_\sppi(\gamma,t)\) for \(\hat\ppi\)-a.e.\ \((\gamma,t)\in\mathfrak C(\B)\). Since
\(\B\) is Asplund, we have that \(L^p_{w^*}(\B,\mu;\B^*)\cong L^p(\B,\mu;\B^*)\). For any \(f\in\mathscr A\coloneqq C^1(\B)\cap\LIP_b(\B)\),
\[\begin{split}
|(\e^*\d f)_{|{\bf E}_\sppi}|_{({\bf E}_\sppi)^*}(\gamma,t)
&=\1_{\{\Der_\sppi\neq 0\}}(\gamma,t)\frac{|(\e^*\d f)(\Der_\sppi)(\gamma,t)|}{|\Der_\sppi|_\B(\gamma,t)}
=\1_{\{\Der_\sppi\neq 0\}}(\gamma,t)\frac{|(f\circ\gamma)'_t|}{\|\dot\gamma_t\|_\B}\\
&\leq|\D_\mu f|(\gamma_t)=(|\d_\mu f|_{(T_\mu\B)^*}\circ\e)(\gamma,t)=|(\e^*\d f)_{|\e^*T_\mu \B}|_{(\e^*T_\mu\B)^*}(\gamma,t)
\end{split}\]
for \(\hat\ppi\)-a.e.\ \((\gamma,t)\in\mathfrak C(\B)\). Since \(\{\e^*\d f\,:\,f\in\mathscr A\}\) generates
\(L^p(\mathfrak C(\B),\hat\ppi;\B^*)\) thanks to Lemma \ref{lem:cyl_generators_cotg} and \eqref{eq:gen_pullback},
by using Lemma \ref{lem:technical_mod} we deduce that \({\bf E}_\sppi\preceq\e^*T_\mu\B\), whence
the property \eqref{eq:tg_bundle_span_tp_claim} follows.

We now prove that \(T_\mu\B\) is the minimal element of \((\mathscr D_\mu(\B),\preceq)\) with this property. Let \({\bf E}\in\mathscr D_\mu(\B)\)
be such that \(\Der_\sppi\in\Gamma_q(\e^*{\bf E})\) for every \(\ppi\in\Pi_q(\B,\mu)\). Fix a master \(q\)-test plan \(\ppi\in\Pi_q(\B,\mu)\). Then
\[\begin{split}
|\d_{|T_\mu\B}f|_{(T_\mu\B)^*}(x)&=|\D_\mu f|(x)\overset{\eqref{eq:equiv_master_tp}}=
\1_{{\rm S}_\sppi}(x)\underset{\hat\sppi_x\text{-a.e.\ }(\gamma,t)}{\rm ess\,sup}\1_{\{\Der\neq 0\}}(\gamma,t)\frac{|(f\circ\gamma)'_t|}{\|\dot\gamma_t\|_\B}\\
&=\1_{{\rm S}_\sppi}(x)\underset{\hat\sppi_x\text{-a.e.\ }(\gamma,t)}{\rm ess\,sup}\1_{\{\Der\neq 0\}}(\gamma,t)\frac{|(\e^*\d f)(\Der_\sppi)|(\gamma,t)}{|\Der_\sppi|_\B(\gamma,t)}\\
&\leq\1_{{\rm S}_\sppi}(x)\underset{\hat\sppi_x\text{-a.e.\ }(\gamma,t)}{\rm ess\,sup}|(\e^*\d f)_{|\e^*{\bf E}}|_{(\e^*{\bf E})^*}(\gamma,t)\leq|\d_{|{\bf E}}f|_{{\bf E}^*}(x)
\end{split}\]
holds for \(\mu\)-a.e.\ \(x\in\B\), for every given function \(f\in\mathscr A\). Since \(\{\d f\,:\,f\in\mathscr A\}\) generates \(L^p(\B,\mu;\B^*)\),
we finally conclude that \(T_\mu\B\preceq{\bf E}\) thanks to Lemma \ref{lem:technical_mod}. The validity of Theorem \ref{thm:main_tg_bundle} i) follows.

De facto, the above proof shows that for any master \(q\)-test plan \(\ppi\) on \((\B,\mu)\), the bundle \(T_\mu\B\) is the unique minimal element of \((\mathscr D_\mu(\B),\preceq)\) such that \(\dot\gamma_t\in T_\mu\B(\gamma_t)\) holds for
\(\hat\ppi\)-a.e.\ \((\gamma,t)\in\mathfrak C(\B)\).
\smallskip

{\bf Step 2: proof of Theorem \ref{thm:main_tg_bundle} ii).}\\
Let \(V_\sppi\B\) be as in Lemma \ref{lem:char_master_tp_aux}. Then to prove Theorem \ref{thm:main_tg_bundle} ii) amounts to showing that
\[
V_\sppi\B=T_\mu\B\quad\text{ for every master \(q\)-test plan }\ppi\text{ on }(\B,\mu).
\]
First, let us prove that \(V_\sppi\B\preceq T_\mu\B\). Let \(S_\sppi\B\) be as in Lemma \ref{lem:char_master_tp_aux}. By Proposition \ref{prop:Castaing},
we can find \((v_k)_{k\in\N}\subseteq\Gamma_q(V_\sppi\B)\) such that \({\rm cl}_\B(\{v_k(x)\,:\,k\in\N\})={\rm cl}_\B(S_\sppi\B(x))\) for \(\mu\)-a.e.\ \(x\in\B\).
To prove that \(V_\sppi\B\preceq T_\mu\B\), it suffices to show that \(v_k\in\Gamma_q(T_\mu\B)\) for every \(k\in\N\). We argue by contradiction: suppose there exist
\(k_0\in\N\), a Borel set \(E\subseteq\B\), and \(\delta>0\) such that \(\e_\#\hat\ppi(B)>0\) and \(\|v_{k_0}(x)-T_\mu\B(x)\|_\B\geq\delta\)
for \(\e_\#\hat\ppi\)-a.e.\ \(x\in E\). On the other, by the definition of \(S_\sppi\B\) we have
\[
\hat\ppi_x\Big(\Big\{(\gamma,t)\in\mathfrak C(\B)\;\Big|\;\|\dot\gamma_t-v_{k_0}(x)\|_\B<\delta\Big\}\Big)>0\quad\text{ for }\e_\#\hat\ppi\text{-a.e.\ }x\in E.
\]
By taking the property \eqref{eq:tg_bundle_span_tp_claim} into account, it thus follows that
\[
\e_\#\hat\ppi\Big(\Big\{x\in E\;\Big|\;\|v_{k_0}(x)-T_\mu\B(x)\|_\B<\delta\Big\}\Big)
\geq\hat\ppi\Big(\Big\{(\gamma,t)\in\e^{-1}(E)\;\Big|\;\|\dot\gamma_t-v_{k_0}(\gamma_t)\|_\B<\delta\Big\}\Big)>0.
\]
This leads to a contradiction with our choice of \(k_0\), \(E\), \(\delta\). Therefore, we deduce that \(V_\sppi\B\preceq T_\mu\B\).

We now pass to the verification of \(T_\mu\B\preceq V_\sppi\B\). For \(\e_\#\hat\ppi\)-a.e.\ \(x\in\B\), it holds that
\[
\dot\gamma_t=\Der(\gamma,t)\in{\rm spt}(\Der_\#\hat\ppi_x)=S_\sppi\B(\gamma_t)\subseteq V_\sppi\B(\gamma_t)\quad\text{ for }\hat\ppi_x\text{-a.e.\ }(\gamma,t)\in\mathfrak C(\B).
\]
Therefore, we have that \(\dot\gamma_t\in V_\sppi\B(\gamma_t)\) for \(\hat\ppi\)-a.e.\ \((\gamma,t)\in\mathfrak C(\B)\), so that \(T_\mu\B\preceq V_\sppi\B\)
thanks to the last paragraph of Step 1 of this proof.
\end{proof}
\section{Consistency with the metric vector calculus}
The aim of this conclusive section is to check the consistency of the vector calculus we develop in this paper
with the differential structure for metric measure spaces introduced by Gigli in \cite{Gigli14}.
\medskip

Let \((\B,\mu)\) be a weighted Banach space, \(q\in(1,\infty)\), and \({\bf E}\in\mathscr D_\mu(\B)\). Then the \(q\)-section space \(\Gamma_q({\bf E})\)
is a (complete) \textbf{\(L^q(\mu)\)-normed \(L^\infty(\mu)\)-module}, in the sense of \cite[Definition 1.2.10]{Gigli14}. Its dual \(\Gamma_q({\bf E})^*\cong\Gamma_p({\bf E}_{w^*}^*)\)
is an \(L^q(\mu)\)-normed \(L^\infty(\mu)\)-module, which coincides with the dual \(\Gamma_q({\bf E})\) in the sense of \cite[Definition 1.2.6 and Proposition 1.2.14]{Gigli14}.
Notice that the notion of `generating vector subspace' of \(\Gamma_q({\bf E})\) or \(\Gamma_p({\bf E}_{w^*}^*)\) we introduced in Definition \ref{def:gen_mod} is consistent
with the one of \cite[Definition 1.4.2]{Gigli14}. We refer to \cite{LPV22} for a more detailed discussion on these topics.
\medskip

In \cite{Gigli14}, the language of \(L^p(\mu)\)-normed \(L^\infty(\mu)\)-modules was used to develop a vector calculus for arbitrary metric measure spaces.
In this regard, an object playing a fundamental role is the \textbf{cotangent module}, introduced in \cite[Definition 2.2.1]{Gigli14} (see also \cite[Theorem/Definition 2.8]{Gigli17}
for another axiomatisation and \cite[Theorem 3.2]{GigPas20} for the case \(p\neq 2\)). The cotangent module is canonically associated with an \textbf{abstract differential}
operator. In the next result, we show the consistency of our machinery with Gigli's notions of cotangent module and abstract differential.
\begin{theorem}[Cotangent module]\label{thm:cotg_mod}
Let \((\B,\mu)\) be a weighted Banach space such that \(\B\) is an Asplund space having the Radon--Nikod\'{y}m property. Let \(p\in(1,\infty)\).
Then \(\Gamma_q(T_\mu\B)^*\) is isomorphic to the cotangent module of \((\B,\mu)\) and the associated differential is \(\d_\mu\colon W^{1,p}(\B,\mu)\to\Gamma_q(T_\mu\B)^*\).
\end{theorem}
\begin{proof}
First, \(\{\d_{|T_\mu\B}f\,:\,f\in C^1(\B)\cap\LIP_b(\B)\}\) generates \(\Gamma_q(T_\mu\B)^*\) by Lemma \ref{lem:cyl_generators_cotg}, thus a fortiori
\(\{\d_\mu f\,:\,f\in W^{1,p}(\B,\mu)\}\) generates \(\Gamma_q(T_\mu\B)^*\). Also, \(|\d_\mu f|_{(T_\mu\B)^*}=|\D_\mu f|\) for every \(f\in W^{1,p}(\B,\mu)\)
by \eqref{eq:char_L_f_cl2}. The cotangent module and the asbtract differential are uniquely determined (up to a unique isomorphism)
by this property (see \cite[Theorem 3.2]{GigPas20}), thus the statement is proved.
\end{proof}

Under the assumptions of Theorem \ref{thm:cotg_mod}, the \textbf{tangent module} of \((\B,\mu)\) (which was defined in
\cite[Definition 2.3.1]{Gigli14} as the dual of its cotangent module) can be different from \(\Gamma_q(T_\mu\B)\). Indeed, the \(q\)-section space
\(\Gamma_q(T_\mu\B)\) is the \emph{predual} of the cotangent module instead. However:
\begin{corollary}
Let \((\B,\mu)\) be a weighted Banach space with \(\B\) reflexive and \(p\in(1,\infty)\). Then the cotangent module of \((\B,\mu)\) is reflexive
and \(\Gamma_q(T_\mu\B)\) is isomorphic to the tangent module of \((\B,\mu)\).
\end{corollary}
\begin{proof}
Since \(L^q(\B,\mu;\B)\) is reflexive, its subspace \(\Gamma_q(T_\mu\B)\) is reflexive, so also the cotangent module \(\Gamma_q(T_\mu\B)^*\)
is reflexive. In particular, the tangent module \(\Gamma_q(T_\mu\B)^{**}\) is isomorphic to \(\Gamma_q(T_\mu\B)\).
\end{proof}

We point out that -- as far as we know -- the reflexivity of the cotangent module might not follow directly from that of \(W^{1,p}(\B,\mu)\). It is known that
the reflexivity of the cotangent module implies that of the Sobolev space \cite[Proposition 2.2.10]{Gigli14}, but whether the converse implication holds
is still an open problem, cf.\ with \cite[Remark 2.2.11]{Gigli14}. Finally, another consistency check:
\begin{remark}[Velocity of a test plan]\label{rmk:vel_tp}{\rm
Let \((\B,\mu)\) be a weighted Banach space with \(\B\) reflexive and let \(q\in(1,\infty)\). Given any \(q\)-test plan \(\ppi\in\Pi_q(\B,\mu)\), one can consider
the \textbf{velocity} \(\ppi'\in\Gamma_q(\e^*T_\mu\B)\) of \(\ppi\) in the sense of \cite[Theorem 1.21]{Pasqualetto22} (see \cite[Theorem 2.3.18]{Gigli14} for
the original definition, under extra assumptions). Letting \(\Der_\sppi\) be the \(\hat\ppi\)-a.e.\ equivalence class of \(\Der\colon\mathfrak C(\B)\to\B\), we claim that
\begin{equation}\label{eq:char_Der_ppi}
\Der_\sppi=\ppi'\quad\text{ for every }\ppi\in\Pi_q(\B,\mu).
\end{equation}
Indeed, Step 1 of the proof of Theorem \ref{thm:main_tg_bundle} shows that \(\Der_\sppi\in\Gamma_q(\e^*T_\mu\B)\), and an application of the dominated convergence theorem
ensures that for every \(f\in W^{1,p}(\B,\mu)\) it holds that
\[
\lim_{h\to 0}\bigg\|\frac{f\circ\e_{t+h}-f\circ\e_t}{h}-(\e^*\d f)(\Der_\sppi)(\cdot,t)\bigg\|_{L^1(\sppi)}=0\quad\text{ for }\mathcal L_1\text{-a.e.\ }t\in[0,1],
\]
whence the claimed identity \eqref{eq:char_Der_ppi} follows by the uniqueness part of \cite[Theorem 1.21]{Pasqualetto22}.
\fr}\end{remark}
\small
\def\cprime{$'$} \def\cprime{$'$}

\end{document}